\documentclass[twoside, 11pt, reqno]{article}

\makeatletter

\usepackage{graphicx}   
\usepackage{subcaption} 
\usepackage{tikz}
\usetikzlibrary{arrows.meta, positioning}
\usepackage[numbers]{natbib}
\usepackage{amsfonts}
\usepackage{amsmath}
\usepackage{amsthm}
\usepackage{amssymb}
\usepackage{color}
\usepackage{graphicx}
\usepackage{hyperref}
\hypersetup{
    colorlinks=true, 
    linktoc=all,     
    linkcolor=blue,  
}

\usepackage{comment}
 \usepackage[top=0.75in, bottom=0.75in, left=0.75in, right=0.75in]{geometry}

\usepackage{algorithm}     
\usepackage{algpseudocode} 

\usepackage[shortlabels]{enumitem}
\usepackage{mathrsfs}  
\usepackage{mathtools}
\usepackage{xcolor}
\usepackage{enumitem}

\usepackage{multirow}

\usepackage{dsfont}

\numberwithin{equation}{section}

\theoremstyle{plain}

\newtheorem{theorem}{Theorem}[section]
\newtheorem{lemma}{Lemma}[section]
\newtheorem{prop}{Proposition}[section]
 \newtheorem{corollary}{Corollary}[section]
\newtheorem{assumption}{Assumption}[section]

\theoremstyle{definition}

\newtheorem{remark}[theorem]{Remark}

\numberwithin{table}{section}

\setlist[enumerate]{topsep=4pt,itemsep=0pt,
}

\setlist[itemize]{itemsep=0pt,
topsep=4pt,
}

\newcommand{\cA}{\mathcal{A}}

\newcommand{\cF}{\mathcal{F}}

\newcommand{\cH}{\mathcal{H}}

\newcommand{\cL}{\mathcal{L}}
\newcommand{\cM}{\mathcal{M}}

\newcommand{\cP}{\mathcal{P}}

\newcommand{\cW}{\mathcal{W}}
\newcommand{\cX}{\mathcal{X}}
\newcommand{\cY}{\mathcal{Y}}

\newcommand{\E}{\mathbb{E}}

\newcommand{\R}{\mathbb{R}}

\def \proof{{\noindent \bf Proof. }}
\def \eproof{\hbox{ }\hfill$\Box$}

\newcommand{\set}[1]
    {\ensuremath{\{ #1 \}}}
    
\newcommand{\HP}[1] 
    {\ensuremath{\mathscr{H}^{#1}}}

\newcommand{\esp}[1]{\ensuremath{\mathbb{E} \!\! \left[#1\right] }}

\renewcommand{\Xi}[1]{X_{i #1}}

\newcommand{\vep}{\varepsilon}

\newcommand{\N}{\mathbb{\N}}

\newcommand{\Tr}{\operatorname{tr}}

\newcommand{\cWtwo}{{\mathcal{W}_{2, L^2}}}
\newcommand{\cWinf}{{\mathcal{W}_{2, L^\infty}}}

\newcommand{\statespace}{{\mathbb{R}^d}}

\newcommand{\be}{\begin{eqnarray}}
\newcommand{\ee}{\end{eqnarray}}
\usepackage{bbm}

\def\Terminal{{g}}

\def\a{\alpha}
\def\cA{\mathcal{A}}

\def\cF{\mathcal{F}}

\def\cH{\mathcal{H}}

\def\cL{\mathcal{L}}
\def\cM{\mathcal{M}}

\def\cP{\mathcal{P}}

\def\cW{\mathcal{W}}
\def\cX{\mathcal{X}}

\def\cY{\mathcal{Y}}

\def\d{{\mathrm{d}}}

\def\fb{\psi^W}

\def\bS{{\mathbf{S}}}

\def\sE{{\mathbb{E}}}
\def\sF{{\mathbb{F}}}

\def\sN{{\mathbb{N}}}
\def\sP{\mathbb{P}}

\def\sR{{\mathbb R}}

\def\n{{\mathfrak{n}}}
\def\bmdim{{d^W}}
\def\pathspace{{\mathscr{C}([0,T]; \mathbb{R}^d)}}
\def\upl{{ {}^\ell\! }}

\DeclareMathAlphabet{\mymathbb}{U}{bbold}{m}{n} 

\newcommand{\lx}{{}^\ell\!\mathcal{X}}
\newcommand{\es}{ \bar{\mathcal{X}}}
\newcommand{\ces}{{}^c\!\bar{\mathcal{X}}}
\newcommand{\les}{{}^\ell\!\bar{\mathcal{X}}}

\def\ra{{{\rightarrow}}}

\newcommand{\xinyu}[1]{\textcolor{violet}{\textsf{[Xinyu: #1]}}}
\definecolor{jfc}{HTML}{196f3d} 
\usepackage{titlesec}

\titleformat{\paragraph}[runin]
  {\normalfont\normalsize\bfseries}
  {\theparagraph}
  {1em}
  {}
  [.]
\title{Numerical Approximation for Path-Dependent McKean--Vlasov Control with Non-Asymptotic Error Estimates
}

\author{Olivier Bokanowski\thanks{University Paris Cité, 
Laboratoire Jacques-Louis Lions, {olivier.bokanowski@u-paris.fr}.
This research benefited from the support of the FMJH Program PGMO and from the support to this program from EDF.}
\and
Jean-François Chassagneux\thanks{ENSAE, {jean-francois.chassagneux@ensae.fr}.  This work was supported by the Labex Ecodec (reference project ANR-11-LABEX-0047) and by   the ``Chaire Futures of Quantitative Finance''.}
\and
Xinyu Li\thanks{Corresponding author.} \thanks{University of Oxford, Erlangen AI Hub; xinyu.li@maths.ox.ac.uk.  This work was  supported by EPSRC grant EP/Y028872/1, Mathematical Foundations of Intelligence: An Erlangen Programme for AI.}
\and
Christoph Reisinger\thanks{University of Oxford, {christoph.reisinger@maths.ox.ac.uk}}}

\begin{document}

\maketitle 
\begin{abstract}
Path-dependent McKean--Vlasov (MKV) control models large interacting populations with 
history-dependent dynamics and costs. 
This paper develops a unified approximation-and-learning framework for
continuous time path-dependent MKV problem 
under open-loop controls. First, an Euler 
discretization scheme with piecewise-constant controls is shown to achieve a non-asymptotic 
 error of $O(h^{1/4})$. Second, we establish a discrete dynamic programming principle 
and prove value equivalence between open-loop and history-dependent feedback controls, 
enabling optimization on a reduced filtration. Third, an interacting particle system is 
introduced to approximate the continuous-time value, yielding an overall error bound of 
$O(h^{1/4}) + O(M^{-\gamma})$ for $M$ particles and an explicitly given $\gamma > 0$.
Finally, we propose a fully implementable 
neural-network policy-gradient method using pathwise features. Numerical experiments, 
including a path-dependent linear-quadratic benchmark, demonstrate the effectiveness of the algorithm.

\end{abstract}

\tableofcontents

\section{Introduction}  
Mean-field control (MFC) problems arise in the optimal control of large populations of weakly interacting agents, whose interactions occur  through the empirical distribution of the population. In the mean-field limit, the collective behaviour is often described by a McKean--Vlasov (MKV) stochastic differential equation (SDE) \cite{mckean1967propagation}, whose coefficients depend on the state and on the law of the population. This framework has applications in mathematical finance, economics, and engineering, including systemic risk \cite{carmona2015systemic,carmona2018probabilistic}, epidemic control \cite{brauer2008mathematical,lee2021controlling}, swarm robotics \cite{elamvazhuthi2019meanfield}, 
and traffic flow \cite{bressan2012nash}; see also \cite{fouque2020deep} for learning-based mean-field methods with delay.


Many practically relevant problems are inherently \emph{path-dependent}: the dynamics and running costs may depend on the entire history of the state trajectory, or on path functionals such as running averages, delayed feedback, or pathwise targets; see, e.g., \cite{ federico2009stochastic,fabbri2017stochastic, guo2026signature, gu2024transportation}. Path-dependence is often
a structural feature of how decisions are made when past information matters beyond the current state. In finance, for example, the payoffs of Asian, lookback, barrier, and autocallable options depend on the price path \cite{grant1997path,vecer2003pricingasian,conze1991path}, as do rough-volatility models \cite{gatheral2018rough} and path-dependent hedging schemes \cite{buehler2019deephedging}. Similarly, in engineering, systems with delays, hysteresis, or integral feedback depend intrinsically on past states \cite{hale1993introduction,fabbri2017stochastic}.

In this paper, we consider a MKV setting where mean-field interactions compound path dependence: each agent's dynamics and costs may depend on the full path of its own state and on the law of the path. This leads to \emph{path-dependent MKV control}, in which both the state dynamics and the cost functional depend on the full path of the controlled process and its law. While the continuous-time theory of such problems has recently been developed \cite{cosso2023optimal}, the design of implementable numerical schemes with rigorous approximation guarantees remains largely open.
\paragraph{Our work}
In this paper, we develop a framework for numerical approximation and learning for path-dependent MKV stochastic control with open-loop controls. We quantify the approximation error and sample complexity of the discrete particle approximation to the continuous-time problem in terms of the time step $h$ and the number of particles $M$, obtaining the rate $\mathcal{O}(h^{1/4}+M^{-\gamma})$ with a parameter $\gamma > 0$.
Our main contributions are as follows:
\begin{itemize}
    \item We first discretise the continuous-time control problem using an Euler--Maruyama scheme with piecewise constant open-loop controls, and establish a non-asymptotic discretisation bound of order $h^{1/4}$ between the continuous and Euler value functions (Theorem~\ref{thm:V_minus_barV}). 
    \item Next, we prove a discrete dynamic programming principle (DPP) for path-dependent MFC. Using the DPP, we show that the discrete-time value is unchanged when controls are feedback in terms of the initial condition and the discrete Brownian increments, so optimisation can be carried out on a reduced filtration (Theorem~\ref{thm:same-value-on-fdbck-ctrl}). 
    \item We then derive a particle approximation bound in terms of $M$ and combine both estimates to obtain a total error of order $\mathcal{O}(h^{1/4}+M^{-\gamma})$, with explicit dimension dependence in $\gamma$ (Theorem~\ref{thm:final_thm}).
    \item Finally, we propose a neural-network-based policy gradient method with controls parametrised by the Brownian-increment paths or discrete state paths, and validate on a linear-quadratic path-dependent MKV benchmark that the optimal numerical value is close to the analytically computed one.
\end{itemize} 

In contrast to the closed-loop setting, where one can assume regularity of the feedback map \cite{soner2025learning}, we impose structural conditions on the drift and running cost rather than additional regularity on the control.
Our central structural assumption (Assumption~\ref{ass assump:Hstruct}) is motivated by Filippov’s convexity condition \cite{Filippov:1962}.
Similar conditions have been used in deterministic optimal control \cite{BokanowskiGammoudiZidani:2022}; here, we formulate a stochastic, path-dependent analogue in Assumption~\ref{ass assump:Hstruct}. This assumption is key to our nonasymptotic analysis in the time step $h$, where we quantify the difference in the induced drifts and running costs under open-loop controls adapted to the continuous and discrete filtrations.
Specifically, it allows an open-loop control path on each time subinterval to be replaced by a constant control without increasing the averaged running cost or changing the drift value (Remark~\ref{rmk:construct_of_a}). 

 Moreover, our convergence analysis does not rely on the study of the infinite dimensional PDE satisfied by the value function, which is particularly challenging in this path-dependent setting.

\paragraph{Related work on learning methods for mean-field control}
The theory of mean-field control has generated a vast literature over the last decade \cite{huang2006large, bensoussan2013mean,carmona2018probabilistic,pham2018bellman,pham2016discrete,yong2013linear}; we refer to \cite{bensoussan2013mean,carmona2018probabilistic} for comprehensive treatments.
Our work addresses the numerical solution of such MFC problems, with a focus on path-dependence.

One line of numerical work for MFC uses deep learning for forward--backward SDE or Bellman/Master formulations, often with specialised architectures \cite{han2024learning, germain2022deepsets, germain2022numerical, reisinger2024fast, pham2022mean,  dayanikli2024deep, carmona2021deep}. See also \cite{10.1214/18-AAP1429} where an iterative solver is studied to approximate MKV FBDSEs for MFC.
A second line approximates the mean-field system by finitely many interacting particles and solves the resulting control problem by machine learning \cite{dayanikli2024deep, fouque2020deep, carmona2022convergence, pham2022mean, picarelli2025extended}, which is closer to our approach.
Carmona and Laurière \cite{carmona2022convergence} obtain non-asymptotic guarantees for finite-horizon MKV control with Lipschitz Markovian feedback, while
\cite{dayanikli2024deep,soner2025learning} assume Lipschitz feedback policy in state and law and establish particle approximation bounds; \cite{reisinger2025convergence} establishes time-discretisation bounds for extended linear--convex MFC under smooth, temporally H\"older-optimal controls.

In contrast, as Markovian policies are not optimal in our path-dependent setting (Section~\ref{ssec:numeric_compare}), we optimise over open-loop controls without regularity restrictions on the control set.
We quantify both the continuous-to-discrete error in terms of the time discretization $h$ and the particle approximation error in terms of the number of particles $M$. 
For model-free reinforcement learning in MFC, we refer to \cite{pham2022mean, pham2023actor, meunier2026model, frikha2025actor, gu2021mean}; we do not pursue this direction here.

\paragraph{Related work on path-dependent MFC}
The seminal paper
 \cite{cosso2023optimal} initiated the theoretical study of path-dependent MFC: well-posedness, a dynamic programming principle in infinite dimension, and law invariance of the value function.
Our focus is complementary, namely discrete approximation, particle systems, and implementable learning schemes.

\paragraph{Organization}
Section~\ref{sec:setup} introduces the path-dependent MKV control problem and the continuous-time formulation. Section~\ref{sec:Euler} develops the Euler discretisation and establishes the main approximation results. Section~\ref{sec:b_feedback} proves the equivalence between open-loop, state-feedback, and Brownian-increment feedback controls. Section~\ref{sec:fully_implement_scheme}
analyses the particle system approximation and introduces the policy gradient algorithm.  Section~\ref{sec:numerics} presents numerical experiments. Detailed proofs are collected in Section~\ref{sec:proofs}.

\paragraph{Notation} 
The set of positive integers is denoted $\mathbb{N}^+$. For $N \in \mathbb{N}^+$, we denote $[N]=\set{1,\dots,N}.$
For $d$ and $m$ in $\mathbb{N}^+$, let $|\cdot|$ denote the Euclidean norm on $\statespace$ and $\|\cdot \|_F$ the Frobenius norm on $\mathbb{R}^{d \times m}$.
Let $\delta_x$ be the Dirac measure at $x \in \statespace$.
Given $T > 0$, $\pathspace$ is the set of continuous $\statespace$-valued functions on $[0, T]$.
For $x \in \pathspace$ and $t \in[0, T]$, write $x_t$ for the value at $t$ and $x_{\cdot \wedge t}:= (x_{s \wedge t})_{s \in[0, T]}$.
We endow $\pathspace$ with $\|x\|_\infty =  \sup _{s \in[0, T]}|x_s|$ and use the path $L^2$-norm $\|x \|_{2}  = \big(\frac{1}{T} \int_0^T |x_s|^2 \,\d s\big)^{1/2}$.
Let $\delta_0$ denote the Dirac measure on $\pathspace$ supported on the constant path~$0$.

\section{Path-dependent McKean--Vlasov Control}\label{sec:setup}

\subsection{Problem setup}

Let $T>0$ be a finite time horizon.
We consider a complete probability space $(\Omega, \mathcal{F}, \mathbb{P})$ that supports a $\bmdim$-dimensional Brownian motion $W = (W_t)_{t \geq 0}$, for some positive integer $\bmdim$. 
We let $\mathbb{F}^W=\left(\mathcal{F}_t^W\right)_{t \geq 0}$ be the $\sP$-completion of the filtration generated by $W.$ 
The filtration $\mathbb{F}^W$ satisfies the usual conditions of $\mathbb{P}$-completeness and right-continuity. 
Given a random variable $\xi$, $\cL(\xi)$ denotes the law of $\xi$.
 Given any $q\geq 2,$
we denote by $\cP_{\infty,q}$   the set of all probability measures $\mu$ on $\pathspace$ with finite $q$th-order moment, i.e., 
\(
\cP_{\infty,q} \coloneqq \{ \mu \in \cP(\pathspace): \int_\pathspace \|x\|_\infty^q \mu (\d x) < + \infty \}.
\)

\begin{assumption}[Standing Assumption]\label{assum:standing} Suppose $q \geq 2.$ Let $\mathcal{G} \subset \mathcal{F}$ satisfy the following conditions:

 (i) $\mathcal{G}$ and $\mathcal{F}_{\infty}^W$ are independent;

 (ii) $\mathcal{G}$ is ``rich enough'' in the sense that 
for every $\mu \in \mathcal{P}_{\infty,q}(\pathspace)$, there exists a continuous and $\mathcal{B}([0, T]) \otimes \mathcal{G}$-measurable process $\xi:[0, T] \times \Omega \rightarrow \statespace$, satisfying $\mathbb{E} \left[ \|\xi\|_\infty^q \right]<\infty$, such that $\xi$ has law equal to $\mu$.
\end{assumption}
According to \cite[Lemma 2.1]{cosso2023optimal} and \cite[Remark 2.1]{cosso2019zero}, Assumption \ref{assum:standing}$(ii)$ implies that we can assume without loss of generality that $\mathcal{G}$ is generated by a uniform  random variable $\Gamma^\mathcal{G}$, meaning any $\mathcal{G}$-measurable random variable can be represented as $\xi = \eta(\Gamma^\mathcal{G})$ a.s.\ for a Borel function $\eta$, which makes $\mathcal{G}$ rich enough to realize any probability measure $\mu \in \mathcal{P}_{\infty,q}(\pathspace)$ via $\mathcal{L}(\eta(\Gamma^\mathcal{G})) = \mu$.

We define the filtration $\mathbb{F} = (\mathcal{F}_t)_{t \geq 0}$ by
\(
\mathcal{F}_t := \mathcal{G} \vee \mathcal{F}_t^W,
\)                                                                          which satisfies the usual conditions of  completeness and right-continuity. We then denote by $\mathbf{S}_q(\mathbb{F})$ (resp. $\mathbf{S}_q(\mathcal{G})$) the set of $\statespace$-valued continuous $\mathbb{F}$-progressively measurable (resp. $\mathcal{B}([0, T]) \otimes \mathcal{G}$ measurable) processes $\xi$ such that
$
\|\xi\|_{\mathbf{S}_q}:=\mathbb{E}\left[\|\xi\|_\infty^q\right]^{\frac{1}{q}}<\infty .
$

\paragraph{Space of probability measures and Wasserstein distance.} 
Let  $(\pathspace,\mathbf{d})$ be a metric space, where $\mathbf{d}$ denotes a given metric on $\pathspace$. For a real number $q\in[2,\infty)$, 
we endow $\cP_q(\pathspace)$ with the topology induced by the Wasserstein metric of order $p$ with $1 \leq p < q$:
\begin{align*}
\mathcal{W}_{p, \mathbf{d}}(\mu, \nu)=\inf \bigg\{\bigg(\int_{\pathspace^2}\mathbf{d}(x,y)^p \eta(\d x, \d y)\bigg)^{1/p}: \eta \in \mathcal{H}(\mu, \nu)\bigg\},
\end{align*}
where $\mathcal{H}(\mu, \nu)$ is the set of all probability measures on $\pathspace \times \pathspace$ with marginals $\mu$ and~$\nu$. Note that if $(\pathspace, \mathbf{d})$ is a Polish space, then the space $(\cP_q(\pathspace), \cW_{p, \mathbf{d}})$ also turns out to be a Polish space \cite[Theorem 6.18]{villani2008optimal}. 
We also include a useful inequality that bounds the Wasserstein distance between the law of two random variables in $\bS_q(\sF)$ with $q \geq 2$:
for any $\cX, \cY: \Omega \times [0,T] \ra \R^d,$
\begin{equation}\label{eq:wasserstein}
    \begin{aligned}
    \cW_{2, \mathbf{d}} \left(\cL(\cX), \cL(\cY) \right)^q & = \inf_{\eta \in \cH(\cL(\cX), \cL(\cY))} \bigg\{
\int_{\pathspace^2} \mathbf{d}\left(x,y \right)^{2} \eta(\d x, \d y)
    \bigg\}^{q/2}\\
    & \leq \sE \left[ \mathbf{d}(\cX, \cY)^{{2}}\right] ^{q/2} 
    \leq \sE \left[\mathbf{d}(\cX,\cY)^{{q}}\right].
\end{aligned}
\end{equation}

\begin{remark}
$(\pathspace,\|\cdot\|_{\infty})$ is Polish; $(\pathspace,\|\cdot\|_{2})$ is not complete, but its completion $\mathscr{L}^2([0,T];\statespace)$ is.
We use $\|\cdot\|_\infty$ for well-posedness of the continuous SDE \cite[Proposition~2.8]{cosso2023optimal} and $\|\cdot\|_2$ for Euler and Gronwall estimates, where exchanging expectation and supremum is not needed.
\end{remark}


\subsection{Continuous-time stochastic control problem}
\label{sec:continuous_sde}
In this section, we formulate the continuous-time stochastic control problem in which the state evolves according to a path-dependent McKean–Vlasov SDE. We first introduce the assumptions that ensure well-posedness, including existence, uniqueness, and moment estimates for the solution. We then define the associated objective functional and the corresponding lifted value function. We follow here \cite{cosso2023optimal} and adapt it to our finite dimensional framework in $\sR^d$, dictated by our goal to provide a numerical method to approximate path-dependent McKean--Vlasov control problems. The study of the Hilbert space case considered in \cite{cosso2023optimal} is left for future research.
\color{black}

\paragraph{Controlled path-dependent McKean--Vlasov SDE}
We consider a finite-horizon stochastic control problem over the time interval $[0, T]$, with fixed terminal time $T > 0$. Let $(A, d_A)$ be a non-empty Polish space. Define $\mathcal{A}$ as the collection of $\mathbb{F}$-progressively measurable processes \mbox{$\alpha : [0, T] \times \Omega \to A$.}

Let $q \geq 2$.
For any given initial time $t \in [0, T]$, initial path $\xi \in \mathbf{S}_q(\sF)$ 
and control process $\a \in \cA$, the system dynamics 
$\cX^{t,\xi, \a}= \left(\cX_s^{t,\xi, \a}\right)_{s \in[0, T]}$ follows a controlled path-dependent McKean--Vlasov stochastic differential equation, where the coefficients depend on the path trajectory and its distribution, as well as the control:
\begin{equation}\label{eq:state_dynamics_dX}
         \begin{aligned}
\cX_s&=\xi_{s \wedge t} + \int_t^{s \vee t}  b_r\left(\cX, \mathbb{P}_{\cX _{\cdot \wedge r}}, \alpha_r \right) \d r 
+\int_t^{s \vee t} \sigma_r\left(\cX, \mathbb{P}_{\cX _{\cdot \wedge r}} \right) \d W_r \quad \forall s \in[0, T], ~\mathbb{P}\text{-a.s}.
\end{aligned}
\end{equation}

Throughout the paper, whenever there is no ambiguity, we drop the initial-time dependence at $t=0$ and write $\mathcal{X}^{0,\xi,\alpha}$ and $\mathcal{X}^{\xi,\alpha}$ interchangeably.

We next present conditions that guarantee the existence and uniqueness of the solutions 
to~\eqref{eq:state_dynamics_dX}.

\begin{assumption}[Lipschitz regularity]\label{ass assum:L2_lip_b_sigma} The functions
$
 b:[0,T]\times \pathspace\times \mathcal P_{\infty,2}\times A 
\longrightarrow \statespace,$ and $
\sigma:[0,T]\times \pathspace\times \mathcal P_{\infty,2}
\longrightarrow\mathbb R^{d\times \bmdim},$
are measurable and satisfy that there exists a constant $C\geq 0$ such that
for all $a\in A$ and for all
$(t, x, \mu),\,\left(t',x^{\prime}, \mu^{\prime}\right)  \in 
 [0,T]\times \pathspace \times \mathcal{P}_{\infty,2}$:

\begin{enumerate}[(i)]
    \item $\left|b(t,x, \mu, a)-b\left(t, x^{\prime}, \mu^{\prime}, a\right)\right|
        \leq C\left(| t-t'|^{\frac{1}{4}} +
          \left\|x_{\cdot\wedge t}-x^{\prime}_{\cdot\wedge t}\right\| _2 + \cWtwo\left(\mu, \mu^{\prime}\right)\right)$;
    \item
        $\left\|\sigma(t,x, \mu)-\sigma\left(t,x^{\prime}, \mu^{\prime}\right)\right\|_{F}  \leq C\left( |t-t'|^{\frac{1}{4}} +  \left\|x_{\cdot\wedge t}-x^{\prime}_{\cdot\wedge t}\right\| _2+\cWtwo\left(\mu, \mu^{\prime}\right)\right)$; 
    \item
    $\left|b\left(t, 0, \delta_0, a\right)\right|+\left\|\sigma\left(t, 0, \delta_0 \right)\right\|_{F} 
     \leq C$.
\end{enumerate}
\end{assumption} 

\color{black}

\begin{remark}
The Lipschitz continuity of $b$ and $\sigma$ with respect to the variable $x \in \pathspace$ 
implies  the non-anticipative properties:
$
b(t,x, \mu, u)=b\left(t,x_{\cdot \wedge t}, \mu, u\right), $ and $ \sigma(t, x, \mu)=\sigma\left(t, x_{\cdot \wedge t}, \mu\right)
$
for every $(t, x, \mu, u) \in[0, T] \times \pathspace \times \cP_{\infty,2} \times A$.
\end{remark}


In the following proposition, we collect several useful properties of the controlled path-dependent McKean–Vlasov SDE.

\color{black}

\begin{prop}\label{prop:X_square_bound}
Let $q\geq 2$, the initial time $t$ be in $[0, T]$ and the initial path $\xi$ be in $\mathbf{S}_q(\sF)$.
Suppose Assumption~\ref{ass assum:L2_lip_b_sigma} holds. Then the following statements hold:

\begin{enumerate}[(i)]
    \item  The SDE \eqref{eq:state_dynamics_dX} admits 
a unique solution $\cX^{t, \xi, \a} \in \mathbf{S}_q(\sF)$. 
    \item The map $[0,T] \times \mathbf{S}_q(\sF), (t,\xi) \mapsto \cX^{t,\xi,\a}$ is jointly continuous in $(t,\xi)$,
and there exists a constant $C_q$ independent of $t,\xi,\a$ such that
\begin{equation}\label{eq:bound_X_square}
    \left\|\cX^{t, \xi, \alpha}\right\|_{\mathbf{S}_q} \leq C_q\left(1+\left\|\xi_{\cdot \wedge t}\right\|_{\mathbf{S}_q}\right).
\end{equation}
    \item   For any $h \geq 0$ and $u\in[t,T]$, there exists a constant $C_q$ such that, 
for any  $h'\in\left[0, h\right]$ the following bound holds:
\begin{equation}\label{eq:Xt_minus_Xn}
    \mathbb{E}\left[\sup_{h' \in [0, h ]}
    \left|\cX^{t,\xi,\alpha}_{(u + h') \wedge T}-\cX^{t,\xi,\alpha}_{u}\right|^q\right]
       \leq C_q\left(1+ \|\xi_{\cdot \wedge t}\|_{\bS_q}^q\right) h^{\frac{q}{2}} .
\end{equation}
\end{enumerate}

\end{prop}
The existence, uniqueness, and joint-continuity of the solution in $(i)$ and the bound in $(ii)$ are established 
in \cite[Proposition~2.8]{cosso2023optimal}
in the case $q=2$. 
The extension to $q\ge 2$ is classical: the proof uses the Burkholder-Davis-Gundy (BDG)
and Hölder inequalities in $L^2$, and the estimates hold also for $L^q$ when $q\ge2$, so the proofs of (i) and  (ii) are omitted.  The proof of (iii) is postponed to the appendix (cf. Appendix \ref{ssec:prop:X_square_bound}.)

\paragraph{Objective functional $J$ and lifted value function $V$}
We now formulate the path-dependent McKean–Vlasov control problem. For a given initial time $t\in [0,T]$, 
an admissible control $\alpha \in \mathcal{A}$ and an initial condition $\xi \in \mathbf{S}_q(\mathbb{F})$
(for some $q \geq 2$), the objective functional $J$ is defined by
\begin{equation}\label{eq:objective_J}
    J(t, \xi, \alpha) = \mathbb{E} \left[ \int_{t}^T  f \left(s, 
    \cX_{\cdot \wedge s}^{t, \xi, \alpha}, \cL(\cX_{\cdot \wedge s}^{t, \xi, \alpha}), 
    \a_s\right)\d t+  \Terminal\left(\cX^{t, \xi, \alpha},\mathcal{L}(\cX^{t, \xi, \alpha})\right) \right],
\end{equation}
where the payoff functions $f$ and $\Terminal$ satisfy the following conditions.

\begin{assumption}\label{assum:terminal_cost_f_g_Phi} 
Let $f : [0,T] \times \pathspace\times \mathcal{P}_{\infty,2} \times A \to \mathbb{R}$ and $\Terminal: \pathspace \times \mathcal{P}_{\infty,2} \to \sR$ denote the running and terminal cost functions respectively.
We impose the following regularity conditions: 

(i)
    The function $f$ and $\Terminal$ are locally Lipschitz continuous in $(x, \mu)$: there exists a constant $C\geq 0$ such that, for every $t,t' \in [0,T]$ and $a \in A$, 
    for all $x,x' \in \pathspace$ and $\mu,\mu' \in \mathcal{P}_{\infty,2}$:
\begin{align*}
   &\big| f(t, x, \mu, a) - f(t', x', \mu', a) \big|
   \leq C \big( 1+ \|x_{\cdot\wedge t}\|_{\infty} + \|x'_{\cdot\wedge t}\|_{\infty}
         +\cWinf(\mu, \delta_0)  \\
   &\qquad \qquad \qquad \qquad  + \cWinf(\mu', \delta_0) \big) \cdot \big( |t-t'|^{\frac14} +  \|x_{\cdot\wedge t}-x'_{\cdot\wedge t}\|_{2} + \cWtwo(\mu,\mu') \big), \\
   &\left| \Terminal(x,\mu) - \Terminal(x',\mu') \right|
   \leq C \big( 1+ \|x\|_{\infty} + \|x'\|_{\infty}
         +\cWinf(\mu, \delta_0) + \cWinf(\mu', \delta_0) \big) \\
   &\quad\cdot \big(  \|x-x'\|_{2} + \cWtwo(\mu,\mu') \big).
\end{align*}
    
(ii) There exists a constant $C \geq 0$ such that 
$ |f(0, 0, \delta_0, a)| + |\Terminal( 0,  \delta_0) | \leq C, 
$ for any $a \in A$. 

\end{assumption}

\begin{remark}\label{rmk:continuity_of_fg}
Assumption~\ref{assum:terminal_cost_f_g_Phi}(i) implies that
$f$ (resp.\ $g$) is locally uniformly continuous in $(x,\mu)$,
uniformly with respect to $(t,a)$ (resp.\ $a$), under both the $L^\infty$ and $L^2$ norms. 
\end{remark}

Under Assumptions \ref{ass assum:L2_lip_b_sigma} and \ref{assum:terminal_cost_f_g_Phi}, from Proposition \ref{prop:X_square_bound} we can get that the objective function $J$ in \eqref{eq:objective_J} is well defined for any $(t,\xi,\alpha) \in [0,T] \times \mathbf{S}_2(\mathbb{F})\times \cA$. We then consider the so-called \textit{lifted value function} 
(cf. \cite{cosso2023optimal, picarelli2025extended}), which is defined as
\begin{equation}
    \label{eq:value_continuous}
    V(t, \xi)=\inf_{\alpha \in \mathcal{A}} J(t, \xi, \alpha), \quad \forall(t, \xi) \in[0, T] \times \mathbf{S}_2(\mathbb{F}).
\end{equation}

\begin{remark}\label{rmk:nonanticipate_fg}
The Lipschitz continuity condition in Assumption \ref{assum:terminal_cost_f_g_Phi}(i) implies that $f$  satisfies the non-anticipativity property:
$
f(t, x, \mu, u) = f(t, x_{\cdot \wedge t}, \mu, u),$
for any $(t,x, \mu, u )\in [0,T] \times \pathspace \times \cP_{\infty,2} \times A.$
Moreover, by Proposition \ref{prop:X_square_bound}, $\cX^{t,\xi,\a}$ only depends on the values of $\xi$ up to time $t, $ i.e.,
$
\cX^{t,\xi, \a} = \cX^{t,\xi_{\cdot\wedge t},\a},
$
then $J$ and $V$ also satisfy the non-anticipativity property:
$$
J(t,\xi,\a) = J(t,\xi_{\cdot \wedge t},\a),\quad V(t,\xi) = V(t,\xi_{\cdot \wedge t}),
$$
for any $(t,\xi)\in [0,T] \times \mathbf{S}_2(\sF) $, and  $ \a \in \cA.$
\end{remark}

 We next present a law invariance property as a direct result of \cite[Theorem~3.6]{cosso2023optimal} and Remark \ref{rmk:continuity_of_fg}, showing that $V$ is invariant with respect to the choice of the random variable $\xi \in \mathbf{S}_2({\sF})$, provided that it has the same distribution. 
\begin{lemma}[Law invariance property]\label{lemma:law_invariance}
    Under Assumptions~\ref{ass assum:L2_lip_b_sigma}  and \ref{assum:terminal_cost_f_g_Phi}, for every $t \in [0,T]$ and $\xi, \xi^{\prime} \in \mathbf{S}_2(\sF)$ with $\mathcal{L}(\xi)=\mathcal{L}\!\left(\xi^{\prime}\right)$, it holds that
    \(
         V(t, \xi) =  V\!\left(t, \xi^{\prime}\right).
    \)
\end{lemma}

As a consequence, we can define the intrinsic value function (with slight abuse of notations) that
${{V}}: [0,T] \times \cP_{\infty,2}  \to \mathbb{R}$ by 
    \(
        {{V}}(t, \mu) \coloneqq  V(t, \xi),  
    \)
for any $\xi \in \mathbf{S}_2(\sF)$ such that $\mathcal{L}(\xi)=\mu \in \cP_{\infty, 2}$.

\section{Discrete-time control problem}\label{sec:Euler}
We now consider a discrete-time approximation of the control problem introduced in the previous section. In this section, we introduce several Euler-type approximations and the corresponding control problems. Our aim is to establish an error bound of order $O(h^{1/4})$ for these discretised control problems with time step size $h$.

We consider a discrete-time grid
\(
\pi := \{0 = t_0 < t_1 < \cdots < t_N = T\}
\)
on the interval $[0,T]$. Without loss of generality, we assume that the grid is \emph{uniform}, i.e., 
$t_n = n h$ for $n = 0,\ldots, N$ with time step $h = T/N$.  
We now consider the discrete counterpart of continuous filtration $\mathbb{F}$, namely 
\(
\mathbb{F}^\pi := (\cF_{t_n})_{0\le n \le N}.
\)
In analogy with the continuous-time case, we introduce the class of discrete-time processes:
\[
\mathbf{S}_q(\mathbb{F}^\pi) := \big\{ \xi \text{ is } \mathbb{F}^\pi \text{ adapted and } \mathbb{E}\big[\max_{n \in [N]} |\xi_n|^q \big] < +\infty\big\}.
\]
for $q \ge 2$.
We denote by $\mathcal{A}^\pi$ the set of discrete-time controls $\bar{\alpha} = \{\bar{\alpha}_{n}\}_{n=0}^{N-1}$ taking values in $A$ and such that $\bar{\alpha}_{n}$ is $\mathcal{F}_{t_n}$-measurable for each $n=0,\dots,N-1$.

We consider the projection operator
\(
\mathfrak{p}^\pi : \pathspace \to (\R^d)^{{N+1}}
\)
and the linear interpolation operator
\(
\ell^\pi : (\R^d)^{{N+1}} \to \pathspace,
\)
both defined on the grid $\pi$.
More precisely, to each continuous path $x \in \pathspace$, we associate the discrete sequence
\(
\mathfrak{p}^\pi(x) = (x_{t_n})_{n=0}^N \in (\R^d)^N.
\)
To alleviate notation, we write ${}^\pi x := \mathfrak{p}^\pi(x)$.
Conversely, for a discrete sequence $(\bar{x}_{t_n})_{0 \le n \le N}$, the operator $\ell^\pi$ returns the path obtained by linear interpolation between the points $\bar{x}_{t_n}$, $0 \le n \le N$. We write
\(
{}^\ell \bar{x} := \ell^\pi\big((\bar{x}_{t_n})_{0 \le n \le N}\big).
\)
With a slight abuse of notation, for a continuous path $x$, we also write ${}^\ell x$ to denote its piecewise linear interpolation $\ell^\pi \circ \mathfrak{p}^\pi(x)$.
Throughout the paper, we also denote by $\n:[0,T]\to\{t_0,\dots,t_N\}$ the function that returns 
the time index $\n(s)=t_n$ whenever $s\in[t_n,t_{n+1})$, and 
 $\n(t_N)=t_N$. For a discrete process $\bar x \in(\mathbb{R}^{d})^{{N+1}}$, we define the stopped process $\bar x_{\cdot \wedge t_n}$ by
 \(
 (\bar x_{\cdot \wedge t_n})_k :=
 \begin{cases}
 \bar x_{t_k}, & 0 \le k \le n\\
 \bar x_{t_n}, & k > n
 \end{cases}
 \).

The following lemma provides useful estimates related to the projection and linear interpolation operators, whose proof is deferred to Section \ref{ssec:proof_le prop interpol}.
\begin{lemma}\label{le prop interpol} 
(i) For $x \in \pathspace$,  it holds that
$\int_0^ {kh}|{}^\ell\!x_t|^2 \d t \le h\sum_{i =0}^k|x_{t_i}|^2.$
 
(ii) For $\xi \in \mathbf{S}_2(\mathbb{F}) $, $\alpha \in \cA$,
   $ \sup_{t\in[0,T]} \esp{\left\|{}^\ell\!\mathcal{X}^{0,\xi,\alpha}_{\cdot\wedge t}
- \mathcal{X}^{0,\xi,\alpha}_{\cdot\wedge t} \right\|_2^2 } \le Ch.$
\end{lemma}

\subsection{Euler scheme update}
We next introduce the Euler scheme update for the continuous state dynamics \eqref{eq:state_dynamics_dX}.
For $n <N$, $\bar{\alpha} \in \cA^\pi$ and a discrete-time initial condition $\xi \in \mathbf{S}_2(\mathbb{F}^\pi)$, we consider the discrete-time process $\bar{\mathcal{X}}^{t_n,\xi,\bar{\alpha}}$ defined on $\pi$ by
\begin{equation}\label{eq de euler scheme}
    \begin{aligned}
    &\bar{\mathcal{X}}^{t_n,\xi,\bar{\alpha}}_{t_k} = \xi_{k}, \quad \text{ for }  k\le n; \\
    &\bar{\mathcal{X}}^{t_n,\xi,\bar{\alpha}}_{t_{k+1}}
    =
    \bar{\mathcal{X}}^{t_n,\xi,\bar{\alpha}}_{t_{k}}
    +
     b(t_k,{}^\ell\!\bar{\mathcal{X}}^{t_n,\xi,\bar{\alpha}}_{\cdot \wedge t_k} ,\cL\!\left({}^\ell\!\bar{\mathcal{X}}^{t_n,\xi,\bar{\alpha}}_{\cdot\wedge t_k}\right),\bar{\alpha}_{t_k}) h \\
     &\qquad\qquad
    + \sigma \left(t_k,{}^\ell\!\bar{\mathcal{X}}^{t_n,\xi,\bar{\alpha}}_{\cdot\wedge t_k} ,\cL\!\left({}^\ell\!\bar{\mathcal{X}}^{t_n,\xi,\bar{\alpha}}_{\cdot\wedge t_k}\right)\right) \Delta W_k, \quad \text{ for } n\le k < N, 
\end{aligned}
\end{equation}
where $\Delta W_k = W_{t_{k+1}} - W_{t_k}$ for any $k\in [N-1]$.

\begin{remark}
The discrete process $\{ \bar{\mathcal{X}}_{t_k}^{t_n, \xi,\bar{\alpha}} \}_{k \in [N]}$ is $\mathbb{F}^\pi$-adapted and belongs to $\mathbf{S}_2(\mathbb{F}^\pi)$. 
The linear interpolation
 $\upl \bar{\mathcal{X}}^{t_n, \xi,\bar{\alpha}}$ is continuous (piecewise linear) but not $\cF_s$-measurable for every $s \geq t_n$.


    
\end{remark}

The following proposition establish the properties of $\bar \cX$, including the flow property and moment estimation bounds:
\begin{prop}\label{pr prop:moment_estimate_first_euler}
Let Assumption \ref{ass assum:L2_lip_b_sigma} hold. 
Fix $q\geq 2$, $\xi \in \mathbf{S}_q(\sF^\pi)$ and $\bar{\alpha} \in \cA^\pi$. 
The process $\bar\cX$ defined in \eqref{eq de euler scheme} 
satisfy the following properties:
\begin{enumerate}[(i)]
    \item For any $t \in \pi$, $\bar \cX^{t, \xi, \bar\a} \in \mathbf{S}_q(\sF^\pi)$.
Moreover, $\bar \cX^{t,\xi,\bar\a}= \bar \cX^{t, \xi_{\cdot \wedge t}, \bar\a}$.
\item The flow property holds, for $t,s \in \pi$ s.t. $t\ge s$:
  $ \bar \cX^{t, \xi, \bar\alpha}= \bar \cX^{s, \bar \cX^{t, \xi, \bar\alpha}, \bar\alpha}. $
\item There exists a constant $C_q$ independent of $t,\xi,\a$ such that
\begin{equation}\label{eq eq:bound_X_square_euler}
    \sE \Big[\max_{t_i \in \pi}|\bar \cX^{t, \xi, \alpha}_{t_i}|^q \Big] + \sE \Big[ { \sup_{s \in [0,T]}|{}^\ell\bar \cX^{t, \xi, \alpha}_s|^q } \Big] \le C_q\Big(1+
    \sE \Big [ {\sup_{i \in [N]: \, t_i \leq t}|\xi_i|^q}  \Big]\Big).
\end{equation}
\end{enumerate}

\end{prop}

\paragraph{The discrete-time optimal control problem}
For $\xi \in \mathbf{S}_2(\mathbb{F}^\pi)$, define
\begin{align}\label{eq objective_bar_J}
    \bar J(t_n, \xi, \bar \alpha) 
    := \E \Bigg[ \sum_{k=n}^{N-1} 
    f\Big(t_k, {}^\ell\!\bar{\mathcal{X}}^{t_n,\xi,\bar{\alpha}}_{\cdot \wedge t_k}, 
    \cL\big({}^\ell\!\bar{\mathcal{X}}^{t_n,\xi,\bar{\alpha}}_{\cdot \wedge t_k}\big),
    \bar{\alpha}_{t_k}\Big) h 
    + \Terminal\Big({}^\ell\!\bar{\mathcal{X}}^{t_n,\xi,\bar{\alpha}}_{\cdot}, 
    \cL\big({}^\ell\!\bar{\mathcal{X}}^{t_n,\xi,\bar{\alpha}}_{\cdot}\big)\Big) 
    \Bigg].
\end{align}
This functional is the discrete-time counterpart of \eqref{eq:objective_J}. We then define the associated value function by
\begin{align}\label{eq barV}
    \bar V(t_n, \xi) := \inf_{\bar{\alpha} \in \mathcal{A}^\pi} \bar J(t_n, \xi, \bar \alpha).
\end{align}

\begin{remark}\label{rmk:nonanticipate_J_V} 
Let $t\in\pi $. The discrete-time process $\bar \cX^{t,\xi,\bar\a}$   depends on the values of $\xi$ up to time $t, $  i.e.,
$
 \bar \cX^{t,\xi, \bar\a} = \bar \cX^{t,\xi_{\cdot\wedge t},\bar \a}.
 $
Since $f$ is non-anticipative (cf. Remark \ref{rmk:nonanticipate_fg}), 
$\bar J$ and $ \bar V,$ also satisfy the non-anticipative property, i.e., 
$$ 
  \bar J(t,\xi, \bar\a) = \bar J(t,\xi_{\cdot \wedge t}, \bar\a),
  \quad 
  \bar V(t,\xi) = \bar V(t,\xi_{\cdot \wedge t}).
$$
\end{remark}
In the following lemma, we establish the continuity of the objective function $\bar J$ and the value function $\bar V$ in the Euler scheme:
\begin{lemma}\label{lemma:continuity_of_barV}
Suppose that Assumptions \ref{ass assum:L2_lip_b_sigma} and \ref{assum:terminal_cost_f_g_Phi} hold. Then $\bar{J}$ and $\bar{V}$ are continuous in $\xi$.
More precisely, for any $t_n \in \pi$, any $\bar{\alpha} \in \mathcal{A}^\pi$, any sequence $\{\xi^{(j)}\}_{j\in\sN} \subset \bS_2(\sF^\pi)$, and any $\xi \in \bS_2(\sF^\pi)$ such that
\(
\lim_{j \ra \infty}\sum_{k \le n} \mathbb{E}\big[|\xi^{(j)}_{k} - \xi_k|^2\big] = 0,
\)
we have
\[
\lim_{j\ra \infty}\bar{J}(t_{n}, \xi^{(j)}, \bar{\alpha}) =\bar{J}(t_{n}, \xi, \bar{\alpha}),
\quad \text{and} \quad
\lim_{j\ra \infty} \bar{V}(t_{n}, \xi^{(j)}) =\bar{V}(t_{n}, \xi).
\]
\end{lemma}

\proof 
Fix $t_n$ in $\pi$ and $\bar \a \in \cA^\pi.$
To show $\bar J$ and $\bar V$ are continuous with respect to given paths, we first show that
$\xi \mapsto \bar \cX^{t_n, \xi, \bar \a}$ is continuous. 

\paragraph{Step 1. Continuity of $\xi \mapsto \bar \cX^{t_n, \xi, \bar \a}$} Fix the discrete processes $ \xi$ and $ \eta$ in $\bS_2(\sF^\pi).$ 
To ease the notation, we denote $ \bar\cX^\xi = \bar\cX^{t_n, \xi, \bar \a}, \bar\cX^\eta =\bar \cX^{t_n, \eta, \bar \a}.$ We also denote $\Delta b_k = b(t_k,{}^\ell\!\bar{\mathcal{X}}^{\xi}_{\cdot \wedge t_k} ,\cL\!\left({}^\ell\!\bar{\mathcal{X}}^{\xi}_{\cdot\wedge t_k}\right),\bar{\alpha}_{k}) - b(t_k,{}^\ell\!\bar{\mathcal{X}}^{\eta}_{\cdot \wedge t_k} ,\cL\!\left({}^\ell\!\bar{\mathcal{X}}^{\eta}_{\cdot\wedge t_k}\right),\bar{\alpha}_{k})$, and similarly for $\Delta \sigma_k.$

By \eqref{eq de euler scheme}, for any $t_k \in \pi$ and $t_k \geq t_n,$
\begin{align*}
    &\sE[|\bar\cX^\xi_{t_{k+1}} - \bar\cX^\eta_{t_{k+1}}|^2] \\
    &= \sE[|\bar\cX_{t_k}^\xi - \bar\cX_{t_k}^\eta|^2] + {2} h \sE[ (\bar\cX_{t_k}^\xi -\bar \cX_{t_k}^\eta)^\top \Delta b_k] + h^2 \sE[|\Delta b_k|^2] + h \sE[|\Delta \sigma_k|^2] \\
    &\leq (1+h) \sE[|\bar\cX_{t_k}^\xi -\bar \cX_{t_k}^\eta|^2] + Ch \sE [\|
    \upl\bar\cX_{\cdot\wedge t_k}^\xi -\upl \bar\cX_{\cdot \wedge t_k}^\eta \|_2^2] \\
    &\leq (1+h)\sE[|\bar\cX_{t_k}^\xi -\bar \cX_{t_k}^\eta|^2]  + C h^2 \sum_{i =0}^k \sE[|\bar\cX_{t_i}^\xi -\bar \cX_{t_i}^\eta|^2] \\
    &= (1+h)\sE[|\bar\cX_{t_k}^\xi -\bar \cX_{t_k}^\eta|^2]  + C h^2 \sum_{i=n+1}^k \sE[|\bar\cX_{t_i}^\xi -\bar \cX_{t_i}^\eta|^2] + C h^2 \sum_{i=0}^n \sE[|\xi_i-\eta_i|^2],
\end{align*}
where the first inequality follows from Assumption \ref{ass assum:L2_lip_b_sigma}, and the second inequality follows from Lemma \ref{le prop interpol} (i).
Therefore,
by summing up the above inequality from $n$ to $k$, and denoting $y_k = \sE[|\bar\cX_{t_k}^\xi -\bar \cX_{t_k}^\eta|^2]  $, we have
\begin{align*}
    y_{k+1} \leq y_n + Ch\sum_{i=n}^k y_i + Ch \sum_{i=0}^n \sE[|\xi_i - \eta_i|^2]
\end{align*}
Hence, by Gronwall's inequality,
$y_k \leq  C (1+Ch)   \sum_{i=0}^n \sE[|\xi_i - \eta_i|^2] .
$
This shows that $\xi \mapsto \bar \cX^{\xi}$ is continuous.

\paragraph{Step 2. Continuity of $\bar V$ and $\bar J$} We begin noticing that, for any $t\in\pi$ and for every $\xi, \eta \in\mathbf{S}_2(\mathbb{F^\pi})$,
$
|\bar V(t, \xi)- \bar V(t, \eta)| \leq \sup _{\bar \alpha \in \mathcal{A}^\pi}|\bar J(t, \xi, \bar \alpha)-\bar J(t, \eta, \bar \alpha)| .
$
Then, the continuity of $\bar V$ follows once we prove that $\bar J$ is continuous in $\xi$ uniformly with respect to $\alpha$.
Such a property is a straightforward consequence of the last statement of Step 1. and of Assumption \ref{ass assum:L2_lip_b_sigma}.
\eproof

\begin{lemma}\label{lemma:discrete_law_invariance}
Suppose that Assumptions \ref{ass assum:L2_lip_b_sigma} and \ref{assum:terminal_cost_f_g_Phi} hold.  For any  $t\in \pi$, any $\xi$ and $\xi'$ in $\bS_2(\sF^\pi)$ such that $\cL(\xi) = \cL(\xi')$, we have
$\bar V(t, \xi)=\bar V(t, \xi'). $ 
\end{lemma}
\begin{remark}\label{re intrinsic value disc time control}
    The above lemma allows us to define the intrinsic value of the discrete time control problem $\bar V(t_n,\mu)$, for $\mu \in \cP_2((\R^d)^N)$.  
\end{remark} 
The proof of Lemma \ref{lemma:discrete_law_invariance} follows the similar lines as that of \cite[Theorem~3.6]{cosso2023optimal} for continuous time MKV problem. Hence, it is omitted here.


%
    

\paragraph{Dynamic programming principle} We next establish the discrete dynamic programming principle (DPP) for the Euler update in \eqref{eq de euler scheme}.

\begin{lemma}[Dynamic programming principle]\label{lemma:DPP}
            Suppose that Assumptions \ref{ass assum:L2_lip_b_sigma} and \ref{assum:terminal_cost_f_g_Phi} hold. Then the lifted value function for the Euler scheme update $\bar \cX$ in \eqref{eq de euler scheme} satisfies the dynamic programming principle (DPP): for every $1 \leq n \leq N$, and every $\xi \in \mathbf{S}_2(\sF^\pi)$, 
\begin{align}\label{eq barV dpp}
    \bar V(t_n, \xi) := \inf_{\bar{\alpha} \in \mathcal{A}^\pi} 
    \mathbb{E} \left[ f \left(t_n,{}^\ell\!\bar{\mathcal{X}}^{t_n,\xi,\bar{\alpha}}_{\cdot} ,\cL\!\left({}^\ell\!\bar{\mathcal{X}}^{t_n,\xi,\bar{\alpha}}_{\cdot}\right),\bar{\alpha}_{n}\right) h
    +  \bar V(t_{n+1}, \bar{\mathcal{X}}^{t_n,\xi,\bar{\alpha}}_{\cdot}) \right]
\end{align}
\end{lemma}

\proof 
We first define for any $t_n \in \pi, \xi \in \bS_2(\sF^\pi)$,
\begin{equation}\label{eq:Lambda}
        \Lambda(t_n, \xi) \coloneqq \inf_{\bar \alpha \in \cA^\pi} 
      \sE \left[f(t_n, \xi, \cL(\xi_{\cdot \wedge t_n}), \bar \a_n) \right] h + \bar V \left(t_{n+1}, \bar \cX_\cdot ^{t_n, \xi, \bar \alpha}\right).
    \end{equation}

\paragraph{Step 1} \textit{Prove the inequality $\Lambda(t_n, \xi) \leq \bar V (t_n, \xi).$}
    Fix $t_n \in \pi$, $\xi \in \bS_2(\sF^\pi)$, and $\bar \alpha \in \cA^\pi$. By \eqref{eq:Lambda},
    \begin{equation*}
        \begin{aligned}
           & \Lambda(t_n, \xi) \leq  \sE \left[f(t_n, \xi, \cL(\xi), \bar \a_n) \right] h + \bar V \left(t_{n+1}, \bar \cX_\cdot^{t_n,\xi, \bar \alpha} \right)
            \\&\leq \sE\Bigg[ f(t_n, \xi, \cL(\xi), \bar \a_n)  h
            + \sum_{i=n+1}^{N-1} f\left(t_{i}, \upl 
           \bar \cX_{\cdot \wedge t_{i}}^{t_{n+1}, \bar \cX^{t_n,\xi,\bar \alpha}, \bar\alpha }, 
           \cL\left( \upl \bar \cX_{\cdot \wedge t_{i}}^{t_{n+1}, \bar \cX^{t_n,\xi,\bar \alpha}, \bar\alpha }\right), 
           \bar \a_i \right) h \\
           &
            \qquad+\Terminal\left(\upl \bar \cX_\cdot^{t_{n+1}, \bar \cX_\cdot^{t_n,\xi,\bar \alpha}, \bar\alpha }, \cL \left(\upl \bar \cX_\cdot^{t_{n+1}, \bar \cX_\cdot^{t_n,\xi,\bar \alpha}, \bar\alpha }\right)\right) \Bigg],
        \end{aligned}
    \end{equation*}
    where the second inequality follows from \eqref{eq barV}.
    By applying the flow property (cf. Proposition \ref{pr prop:moment_estimate_first_euler}(ii)), we have
    $
  \bar \cX_{t_i}^{t_{n+1}, \bar\cX^{t_n, \xi, \bar\alpha}, \bar\alpha} =   \bar \cX_{t_i}^{t_n, \xi, \bar\alpha} , \text{ for any } n+1 \leq i \leq N.
     $
    As $\bar \alpha$ is an arbitrary control in $\cA^\pi$,  we have
    \begin{equation}
    \begin{aligned}
            \Lambda(t_n, \xi) &\leq \inf_{\bar \alpha \in \cA^\pi}  \sE\Big[ 
            \sum_{i=n}^{N-1} f\left(t_{i}, 
         \upl  \bar \cX_{\cdot \wedge t_{i}}^{t_{n}, \xi, \bar\alpha }, 
           \cL\left( \upl \bar \cX_{\cdot \wedge t_{i}}^{t_{n}, \xi, \bar\alpha }\right), 
           \bar \a_i \right) h +\Terminal\left( \upl \bar \cX_\cdot^{t_{n}, \xi, \bar\alpha },\cL \left( \upl\bar \cX_\cdot^{t_{n}, \xi, \bar\alpha }\right)\right) \Big]\\
            & = \bar V\left(t_n, \xi\right).
    \end{aligned}
    \end{equation}

    \paragraph{Step 2} \textit{Prove the inequality $\Lambda(t_n, \xi) \geq \bar V (t_n, \xi).$ }
    Fix $\vep > 0.$
    By \eqref{eq:Lambda},
    we can find $\alpha^\vep \in \cA^\pi$ s.t.
\begin{equation}\label{eq:Lambda_geq_V}
     \Lambda(t_n, \xi) \geq  \sE \left[f(t_n, \xi, \cL(\xi),  \a^\vep_n) \right] h + \bar V\left(t_{n+1}, \bar \cX^{t_n, \xi, \alpha^\vep}_\cdot\right) - \vep. 
\end{equation}
By definition of $\bar V$ in \eqref{eq barV}, there exists ${\zeta^\vep} \in \cA^\pi$ s.t.
\begin{equation}\label{eq:V_geq_Phi}
\begin{aligned}
     \bar V\left(t_{n+1}, \bar \cX^{t_n,\xi,\alpha^\vep}\right) 
    \geq ~&\sE \Bigg[ \sum_{i=n+1}^{N-1} f\left(t_{i}, 
        \upl   \bar \cX_{\cdot \wedge t_{i}}^{t_{n+1}, \bar \cX^{t_n,\xi, \alpha^\vep}, \zeta^\vep }, 
           \cL\left( \upl \bar \cX_{\cdot \wedge t_{i}}^{t_{n+1}, \bar \cX^{t_n,\xi, \alpha^\vep}, \zeta^\vep }\right), 
           \zeta_i^\vep \right) h \\
           &\quad +\Terminal\left(\upl\bar \cX_\cdot^{t_{n+1}, \bar \cX_\cdot^{t_n,\xi,\alpha^\vep}, \zeta^\vep }, \cL \left(\upl \bar \cX_\cdot^{t_{n+1}, \bar \cX^{t_n,\xi,\alpha^\vep}, \zeta^\vep }\right)\right)\Bigg] -\vep .
\end{aligned}
\end{equation}
Now set 
$\gamma^\vep = \alpha^\vep \mathbbm{1}_{[t_0, t_{n+1})} + \zeta^\vep \mathbbm{1}_{[t_{n+1}, t_N]}$ and notice that $\gamma^\vep \in \cA^\pi.$
Using again the flow property (cf. Proposition \ref{pr prop:moment_estimate_first_euler}(ii)), we obtain
$$
\bar \cX_{t_i}^{t_{n+1}, \bar \cX^{t_n, \xi, \alpha^\vep}, \zeta^\vep} = \bar \cX_{t_i}^{t_n, \xi, \gamma^\vep}, \text{ for }  n+1 \leq i \leq N.
$$
Hence, by combining the above equality with \eqref{eq:Lambda_geq_V}, \eqref{eq:V_geq_Phi}, and \eqref{eq barV}, we obtain,
$$
\begin{aligned}
      &\Lambda(t_n, \xi) \geq   \sE \left[f(t_n, \xi, \cL(\xi_{\cdot \wedge t_n}),  \a^\vep_n) \right] h + \bar V\left(t_{n+1}, \bar \cX^{t_n,\xi,\alpha^\vep}\right) -\vep
    \\&\geq  
   \sE\Bigg[ 
            \sum_{i=n}^{N-1} f\left(t_{i}, 
           \bar \cX_{\cdot \wedge t_{i}}^{t_{n}, \xi, \gamma^\vep }, 
           \cL\left(\bar \cX_{\cdot \wedge t_{i}}^{t_{n}, \xi, \gamma^\vep }\right), 
           \gamma^\vep_i \right) h +\Terminal\left(\bar \cX_\cdot^{t_{n}, \xi, \gamma^\vep }, \cL \left( \bar \cX_\cdot^{t_{n}, \xi, \gamma^\vep }\right)\right) \Bigg]- 2\vep\\
    &\geq \bar V(t_n, \xi) - 2 \vep.
\end{aligned}
$$
The claim then follows from the arbitrariness of $\vep.$ \eproof

\subsection{Approximation error}
We quantify the error between the continuous-time value \eqref{eq:state_dynamics_dX} and the Euler value \eqref{eq de euler scheme} under a structural condition on the drift and cost functions. This key assumption is motivated by Filippov's convexity condition \cite{Filippov:1962,BokanowskiGammoudiZidani:2022}: it guarantees the existence of piecewise constant controls that yield the same averaged drift without increasing the running cost (Remark~\ref{rmk:example_of_hstruct}).

\begin{assumption}\label{ass assump:Hstruct}
    For every given pair $(t, x,\mu) \in [0,T] \times \pathspace \times \cP_{\infty, 2}$, 
\begin{equation*}
 \left \{ 
    \begin{pmatrix}
        b(t,x,\mu,a) \\ f(t,x,\mu,a)+\eta
    \end{pmatrix} : \ a \in A \text{ and } \eta  \geq 0  \right \}
    \text{ is a convex set of } \sR^{d+1}.
\end{equation*}
\end{assumption}

\begin{remark} \label{rmk:construct_of_a}
  $(i)$ Fix $(t_n,x,\mu)\in[0,T-h]\times\pathspace\times\cP_2(\pathspace)$. 
Given $\{a_s\}_{s\in [t_n,t_{n+1}]} \subseteq A$, 
Assumption~\ref{ass assump:Hstruct} guarantees that there exists $\bar a \in A$ and $ \eta \geq 0$ such that 
\be
      & \frac{1}{h}  \int_{t_{n}}^{t_{n+1}} b(t_n,x,\mu,a_s) \d s &=b(t_n,x,\mu,\bar a), \\
      \label{eq:b_f_struct:1}
      & \frac{1}{h}\int_{t_{n}}^{t_{n+1}} f(t_n,x,\mu,a_s) \d s &=f(t_n,x,\mu,\bar a) + \eta \geq f(t_n,x,\mu,\bar a).
      \label{eq:b_f_struct:2}
\ee
$(ii)$ Following the previous point, let us consider a process $X \in \mathbf{S}_2(\mathbb{F}^\pi)$ and a control in $\alpha \in \cA$.
  For any 
$0 \le n \le N-1$, there exists $\bar\a_n$, an $\cF_{t_n}$-measurable $A$-valued random variable which satisfies almost surely that

\begin{align}
\label{eq eq:an_construct:1}
&\mathbb{E}\left[
   \frac{1}{h}
   \int_{t_n}^{t_{n+1}} 
    b\!\left(t_n, \upl X_{\cdot \wedge t_n}, 
      \cL\!\left(\upl X_{\cdot \wedge t_n}\right), 
      \alpha_t\right) \mathrm{d} t 
    \,\bigg\rvert\, \mathcal{F}_{t_n} \right]
    =
    b\!\left(t_n, \upl X_{\cdot \wedge t_n}, 
      \cL\!\left(\upl X_{\cdot \wedge t_n}\right), 
      \bar{\alpha}_n\right), \\
      \label{eq eq:an_construct:2}
    &\mathbb{E}\bigg[
    \frac{1}{h}
    \int_{t_n}^{t_{n+1}} 
    f\!\left(t_n, \upl X_{\cdot \wedge t_n}, 
      \cL\!\left(\upl X_{\cdot \wedge t_n}\right), 
      \alpha_t\right) \mathrm{d} t 
    \,\bigg\rvert\, \mathcal{F}_{t_n}\bigg]
    \ge 
    f\!\left(t_n, \upl X_{\cdot \wedge t_n}, 
      \cL\!\left(\upl X_{\cdot \wedge t_n}\right), 
      \bar{\alpha}_n\right).
\end{align}

\end{remark}

In the following remark, we provide two examples
where Assumption \ref{ass assump:Hstruct} is satisfied.
\begin{remark}\label{rmk:example_of_hstruct}
$(i)$ \emph{Linear convex problem.}
Assume $A$ is a convex subset of $\statespace$, $f$ is convex in $a$, and the drift $b$ is affine in $a$:
$
  b(t,y,\nu,a) =   b_1(t,y,\nu) +b_2(t,y,\nu) a.
$
Then it can be checked that Assumption~\ref{ass assump:Hstruct} holds.
 
$(ii)$ \emph{$b(t,x,\mu, A)$ is convex (for all $(t,x,\mu)$), and  $f$ does not depend of the control.} 
  Then $\begin{pmatrix} b \\ f\end{pmatrix} (t,x,\mu,A) $ is a convex set, 
  which implies Assumption~\ref{ass assump:Hstruct}.
\end{remark}

We are now ready to state the main theorem of this section, which establishes that the discretisation error is of order \(h^{1/4}\) between the value function $V$ for the continuous problem \eqref{eq:state_dynamics_dX} and $\bar V$ for the discrete problem with open-loop control \eqref{eq de euler scheme}.
\begin{theorem}\label{thm:V_minus_barV} 
Suppose that Assumptions  \ref{ass assum:L2_lip_b_sigma}, \ref{assum:terminal_cost_f_g_Phi}  and  \ref{ass assump:Hstruct} hold. Let $\mu \in \cP_{\infty,2}$.
 Then there exists a constant $C$ (independent of $\pi$) such that
 \begin{align}
      \left| V(0, \mu) - \bar V(0, \mathfrak{p}^\pi\sharp\mu) \right| \leq C h^{1/4}.  \label{eq:V_minus_barV}
 \end{align}
 \end{theorem}


The proof of Theorem \ref{thm:V_minus_barV}
depends on the following key lemma, which bounds the difference between 
$\cX$ and its Euler update $\bar \cX$.
\begin{lemma}
\label{le lemma:X_bar_minus_X}
Suppose that Assumptions \ref{ass assum:L2_lip_b_sigma} and \ref{ass assump:Hstruct} hold. Given an initial condition $\xi \in \mathbf{S}_2(\sF)$ and a control $\alpha \in \mathcal{A}$, 
  let $\bar\a\in\cA^\pi$ constructed as in   Remark \ref{rmk:construct_of_a}(ii) 
  Then, 
 \begin{equation} \label{eq:X_bar_minus_X}
    \sup _{t \in [0,T]} \mathbb{E}\left[\left|{}^\ell\!\bar{\mathcal{X}}_{t}^{ 0,{}^\pi\!\xi, \bar{\alpha}}-\mathcal{X}_{t}^{0,\xi, \alpha}\right|^2\right] \leq  C \sqrt{h},
\end{equation}
where $C$ does not depend on $\pi$.

On the other hand, given $\bar{\alpha} \in \cA^\pi$, we also denote by $\bar\alpha$ its piecewise-constant extension (observe that $\bar{\alpha} \in \cA$). There exists a constant $C$ independent of $\bar{\alpha}$ and $h$ such that
\begin{equation}\label{eq:X_bar_minus_X2}
    \sup_{t\in[0,T]} \sE\!\left[\left|\cX^{0,\xi,\bar\alpha}_t 
- \upl \bar{\cX}^{0,{}^\pi\!\xi,\bar\alpha}_t\right|^2\right]
\le C\, \sqrt{h}.
\end{equation}

\end{lemma}

\proof We first prove for \eqref{eq:X_bar_minus_X}.
Let $\xi\in\bS_2(\sF)$ and  $\alpha \in \cA$ .
Let $\bar{\alpha} \in \cA^{\pi}$ be the discrete-time control constructed from $\alpha$ according to Remark \ref{rmk:construct_of_a}(ii). 

To prove Lemma~\ref{le lemma:X_bar_minus_X}, we introduce an auxiliary continuous process that coincides with the discrete process $\{ \bar\cX_{t_k}^{t_n, \xi, \bar \a}\}_{k\in[N]}$ in \eqref{eq de euler scheme} on $\pi$ and is $\cF_s$-measurable for any $s \geq t_n$, for a given initial time $t_n$:
for $t \in [t_k, t_{k+1}]$, we set 
    \begin{align}\label{eq de euler scheme cont}
    \bar{\mathcal{X}}^{t_n,\xi,\bar{\alpha}}_{t}
   & = \bar{\mathcal{X}}^{t_n,\xi,\bar{\alpha}}_{t_{k}}
    + b\!\left(t_k,{}^\ell\!\bar{\mathcal{X}}^{t_n,\xi,\bar{\alpha}}_{\cdot \wedge t_k},
      \cL\!\left({}^\ell\!\bar{\mathcal{X}}^{t_n,\xi,\bar{\alpha}}_{\cdot\wedge t_k}\right),\bar{\alpha}_{t_k}\right) (t-t_k)
    \nonumber\\
    &\quad + \sigma \!\left(t_k,{}^\ell\!\bar{\mathcal{X}}^{t_n,\xi,\bar{\alpha}}_{\cdot\wedge t_k},
      \cL\!\left({}^\ell\!\bar{\mathcal{X}}^{t_n,\xi,\bar{\alpha}}_{\cdot\wedge t_k}\right)\right) (W_t - W_{t_k}).
\end{align}
Then by Proposition \ref{prop:X_square_bound}(iii), $(\bar \cX_t^{t_n,\xi,\bar{\alpha}} )_{t\geq t_n}$ satisfies the following property:
\begin{equation} \label{eq:cont_euler_Xt_minus_Xn}
    \max_{k \geq n}\sup_{t\in[t_k,t_{k+1}]}\sE \!\left[\left|  \bar \cX_t^{t_n,\xi,\bar{\alpha}} - \bar \cX_{t_k}^{t_n,\xi,\bar{\alpha}} \right|^2\right]
    \leq C h.
\end{equation}

Let $\cX\coloneqq \mathcal{X}^{0,\xi,\alpha}$ be the solution to \eqref{eq:state_dynamics_dX}, and $\bar \cX \coloneqq \bar{\mathcal{X}}^{0,{}^\pi\xi,\bar{\alpha}}$ the continuous Euler scheme defined in \eqref{eq de euler scheme cont}. We also denote $\delta \mathcal{X}_t=\mathcal{X}_t-\bar{\mathcal{X}}_t$,
$ 
\delta b_t=
b\left(t, \mathcal{X},\cL(\mathcal{X}_{\cdot \wedge t}), \alpha_t\right)
-
b\left( \mathfrak{n}(t), \les
,\cL\left(\les_{\cdot \wedge \mathfrak{n}(t)} 
\right), \bar{\alpha}_{\mathfrak{n}(t)}\right)
$, 
$\delta\sigma_t = \sigma\left(t, \mathcal{X}, \cL\left(\mathcal{X}_{\cdot \wedge t}\right)\right)
-
\sigma\left(\mathfrak{n}(t), \les,\cL\left(\les_{\cdot \wedge \mathfrak{n}(t)}\right)\right).$ 

%

First,  note that for any $t\in[t_n, t_{n+1}],$
\begin{align*}
    \upl \bar\cX_t - \cX_t = (\upl \bar \cX_t -\bar\cX_{t_n}  ) + (\bar\cX_{t_n} - \cX_{t_n}) + (\cX_{t_n} - \cX_t),
\end{align*}
where by Propositions \ref{prop:X_square_bound} and \ref{pr prop:moment_estimate_first_euler}, we have
\begin{equation}\label{eq:decomposition_in_X_minus_X}
  \sE\left[\left|  \upl \bar\cX_t - \cX_t\right|^2\right] \leq  C h +     \max_{0 \le n \le N} \esp{\left|\bar{\mathcal{X}}_{t_n}^{ 0,{}^\pi\!\xi, \bar{\alpha}}-\mathcal{X}_{t_n}^{0,\xi, \alpha}\right|^2}.  
\end{equation}
Therefore, it remains to prove that
\begin{align}\label{eq:discrete_barX_minus_X}
    \max_{0 \le n \le N} \esp{\left|\bar{\mathcal{X}}_{t_n}^{ 0,{}^\pi\!\xi, \bar{\alpha}}-\mathcal{X}_{t_n}^{0,\xi, \alpha}\right|^2} \le C \sqrt{h}.
\end{align}

By applying Ito's formula to $|\delta \cX_t|^2$ and
taking expectation on both sides
we obtain, 
\begin{align}\label{eq starting point}
    \sE\left[\left|\delta \mathcal{X}_{t_{n+1}}\right|^2\right]
    &= \esp{\left|\delta \mathcal{X}_{t_n}\right|^2
   +2  \int_{t_n}^{t_{n+1}} \delta \mathcal{X}_t^{\top} \delta b_t \mathrm{d} t
   + \int_{t_n}^{t_{n+1}}\left\| \delta \sigma_t \right\|_F^2 \mathrm{d} t  }.
\end{align}

\paragraph{Step 1. Bound $\delta \cX_t^\top \delta b_t$}
To analyse $\delta \cX_t^\top \delta b_t$ in \eqref{eq starting point},
we observe that for any $t\in [t_n, t_{n+1})$, 
\begin{align}
\delta \mathcal{X}_t^\top \delta b_t  & =\left(   \cX_t - \cX_{t_n} + \bar \cX_{t_n} - \bar \cX_t + \delta \cX_{t_n}\right)^\top \delta b_t
\nonumber
\\
&=  \left(\mathcal{X}_{t}-\mathcal{X}_{t_n}\right)^\top \delta b_t+\left( \bar \cX_{t_n} -\bar \cX_{t}\right)^\top \delta b_t 
\label{eq eq:bound_dXdb bis 1}
\\
& +\delta \mathcal{X}_{t_n}^\top 
  \left( 
   b\big(t,   \mathcal{X}, \cL\left(\mathcal{X}_{\cdot \wedge t}\right), \alpha_t\big)
  -b\big(t_n, \lx_{\cdot \wedge t_n}, \cL( \lx_{\cdot \wedge t_n}), \alpha_t \big)
  \right) 
\label{eq eq:bound_dXdb bis 2}
\\
&
 + \delta \mathcal{X}_{t_n}^\top 
 \left(b\big( t_n, \lx, \cL( \lx_{\cdot \wedge t_n}), \alpha_t\big)
      -b\big( t_n, \lx_{\cdot \wedge t_n},\cL(\lx_{\cdot \wedge t_n}), \bar{\alpha}_n\big)
 \right)
 \label{eq eq:bound_dXdb bis 3}
 \\
&
 + \delta \mathcal{X}_{t_n}^\top 
 \left( b\big( t_n, \lx_{\cdot \wedge t_n}, \cL( \lx_{\cdot \wedge t_n}), \bar{\alpha}_n\big)
       -b\big( t_n, \les,\cL(\les_{\cdot \wedge t_n}), \bar{\alpha}_n\big)
 \right).
 \label{eq eq:bound_dXdb bis 4}
\end{align} 

 We first analyse the  term in \eqref{eq eq:bound_dXdb bis 1}. By Assumption \ref{ass assum:L2_lip_b_sigma}(iii), we have 
\begin{equation}\label{eq eq:thrid0_dXdb}
\begin{aligned}
   &\sE\left[ \left|\left(\mathcal{X}_t-\mathcal{X}_{t_n}\right)^\top \delta b_t\right|+\left|\left(\bar \cX_t-\bar \cX_{t_n}\right)^\top \delta b_t\right| \right]
\\
& \leq C \sup_{t\in [t_n,t_{n+1}]}\sE \Big[
\big(|\cX_t - \cX_{t_n}| + | \bar \cX_t - \bar \cX_{t_n}|\big)
\big(1 + \| \cX_{\cdot \wedge t}\|_\infty +  \| \les_{\cdot \wedge t_n} \|_\infty \\
&\qquad\qquad\qquad\qquad\qquad\qquad
  + \sE\left[  \| \cX_{\cdot \wedge t}\|_\infty \right]
  + \sE \left[ \| \les_{\cdot \wedge t_n} \|_\infty \right]
\big) \Big] \\
& \leq C \sup_{t\in [t_n,t_{n+1}]}\sE \Big[|\cX_t - \cX_{t_n}|^2
+ | \bar \cX_t - \bar \cX_{t_n}|^2 \Big]^{\frac{1}{2}}
\Big(1 + \sE \Big[\| \cX_{\cdot \wedge t}\|_\infty ^2 \Big]^{\frac{1}{2}}
+ \sE\Big[ \| \les_{\cdot \wedge t_n} \|_\infty^2\Big] ^{\frac{1}{2}}
\Big) \\
&\leq C \sqrt{h},
\end{aligned}
\end{equation}
where the second inequality follows from Cauchy–Schwarz inequality  and the last inequality is a result of
Proposition \ref{prop:X_square_bound} and \eqref{eq:cont_euler_Xt_minus_Xn}. 

We now analyse the term in \eqref{eq eq:bound_dXdb bis 2}. By Young's inequality,
\begin{align}
    &\esp{\delta \mathcal{X}_{t_n}^\top 
      \left(b\big(t, \mathcal{X}, \cL(\mathcal{X}_{\cdot \wedge t}), \alpha_t\big)
           -b\big(t_n, \lx_{\cdot \wedge t_n}, \cL( \lx_{\cdot \wedge t_n}), \alpha_t\big)\right)} \nonumber \\
    & \qquad \le \ 
      C \,\esp{
       |\delta \mathcal{X}_{t_n}|^2 + 
       \left|
       b\big(t, \mathcal{X}, \cL(\mathcal{X}_{\cdot \wedge t}), \alpha_t\big)
      -b\big(t_n, \lx_{\cdot \wedge t_n}, \cL( \lx_{\cdot \wedge t_n}), \alpha_t\big)\right|^2}.
\end{align}
Note that, by Assumption \ref{ass assum:L2_lip_b_sigma}(i),
\begin{align}
    & \esp{  \left|b\big(t, \mathcal{X}, \cL(\mathcal{X}_{\cdot \wedge t}), \alpha_t\big)
                  -b\big(t_n, \lx_{\cdot \wedge t_n}, \cL( \lx_{\cdot \wedge t_n}), \alpha_t\big)\right|^2} \nonumber \\
    & \quad \le \ C\bigg( (t-t_n)^{\frac{1}{2}} + \sE\left[ \left\| \cX_{\cdot \wedge t_n} - \lx_{\cdot \wedge t_n}\right\|_{2}^2 \right] 
+ \cWtwo\left(\cL(\cX_{\cdot \wedge t}), \cL(\lx_{\cdot \wedge t_n})\right)^2 \bigg) \nonumber
\\
&   \quad \le C\bigg(h^{\frac{1}{2}} +
    \sE \left[ \left\| \cX_{\cdot \wedge {t}} - \cX_{\cdot \wedge t_n}\right\|_{2}^2 \right] +
    \sE\left[ \left\| \cX_{\cdot \wedge {t_n}} - \lx_{\cdot \wedge t_n}\right\|_{2}^2 \right]
    \bigg) \le C h^\frac{1}{2},
\end{align}
where the last line follows from \eqref{eq:wasserstein}, Lemma~\ref{le prop interpol}\textup{(ii)}, and Proposition~\ref{prop:X_square_bound}. 

For the term in \eqref{eq eq:bound_dXdb bis 4}, we use Young's inequality and Assumption \ref{ass assum:L2_lip_b_sigma}(i) to get 
\begin{align*}
    & \esp{\delta \mathcal{X}_{t_n}^\top 
     \left( b\big(t_n, \lx_{\cdot \wedge t_n}, \cL( \lx_{\cdot \wedge t_n}), \bar{\alpha}_n\big)
           -b\big(t_n, \les,\cL(\les_{\cdot \wedge t_n}), \bar{\alpha}_n\big) \right)}
    \\
    & \quad\le \ C\,\esp{|\delta \mathcal{X}_{t_n}|^2 +
     \left|  b\big(t_n, \lx_{\cdot \wedge t_n}, \cL( \lx_{\cdot \wedge t_n}), \bar{\alpha}_n\big)
            -b\big(t_n, \les,\cL(\les_{\cdot \wedge t_n}), \bar{\alpha}_n\big) \right|^2 
    } \\
    & \quad \leq  \ C \,\esp{|\delta \mathcal{X}_{t_n}|^2} + C \sE \left[ \left\| 
{}^\ell\!\mathcal{X}_{\cdot\wedge t_n}
- \les_{\cdot\wedge t_n} \right\|_2^2 
\right]\leq 
    C \esp{|\delta \mathcal{X}_{t_n}|^2} +  C h\sum_{k=0}^n \sE \left[|\delta\mathcal{X}_{t_k}|^2\right],
\end{align*}
where the last inequality follows from Lemma~\ref{le prop interpol}\textup{(i)}.

Finally, by denoting
$ I_n = \int_{t_n}^{t_{n+1}}\delta \mathcal{X}_{t_n}^\top 
  ( b\big(t_n, \lx, \cL( \lx_{\cdot \wedge t_n}), \alpha_t\big)
        -b\big(t_n, \lx_{\cdot \wedge t_n},\cL(\lx_{\cdot \wedge t_n}), \bar{\alpha}_n\big)) \d t $, and
from the construction of $\bar{\alpha}_n$, we have 
   $ \sE[I_n \mid \cF_{t_n}] = 0, $ almost surely.
We thus get 
\begin{align}\label{eq eq:dXdb_result}
   \sE \left[ \int_{t_n}^{t_{n+1}}  \delta \cX_t^\top \delta b_t \d t\right]
        &\leq  
           Ch\left( h^\frac12 +  \sE[|\delta \cX_{t_n}|^2] +
           h\sum_{k=0}^n \sE \left[|\delta\mathcal{X}_{t_k}|^2\right]
\right) .
\end{align}

\paragraph{Step 2. Bound $\|\delta \sigma_t\|_F$}
Following similar procedure as in Step 1, we have
\begin{align}
      \esp{ \left\|\delta \sigma_t \right\|_F^2}
& \leq C  \left( h^\frac12 +  \sE[|\delta \cX_{t_n}|^2]  + h\sum_{k=0}^n \esp{|\delta\mathcal{X}_{t_k}|^2}\right) .\label{eq control sigma}
\end{align}

\paragraph{Step 3. Conclusion} We obtain, combining \eqref{eq control sigma} and \eqref{eq eq:dXdb_result} with \eqref{eq starting point},
 \begin{align}
          \sE\left[ \left|\delta \mathcal{X}_{t_{n+1}}\right|^2  \right]\leq &  (1+Ch)\sE\left[\left|\delta \mathcal{X}_{ t_n}\right|^2 \right]+   Ch \left(\sqrt{h} +h\sum_{k=0}^n \esp{|\delta\mathcal{X}_{t_k}|^2}  \right) .
    \end{align}
Denoting $y_n = \max_{0 \le k \le n}\esp{|\delta\mathcal{X}_{t_k}|^2} $, we deduce from the previous inequality
$    y_{n+1} \le (1+Ch)y_n + Ch^\frac32.$ By Gronwall's inequality and the fact that $y_0=0$, this yields $y_N \le Ch^\frac12$. By combining with \eqref{eq:decomposition_in_X_minus_X}, it concludes the proof of \eqref{eq:X_bar_minus_X}.

The proof of \eqref{eq:X_bar_minus_X2} follows from similar (in fact easier) arguments as used in the proof of \eqref{eq:X_bar_minus_X}, hence it is omitted here.
\eproof

\begin{remark}
        The bound in \eqref{eq:X_bar_minus_X2} in Lemma \ref{le lemma:X_bar_minus_X} can be improved under stronger time regularity of the drift in Assumption \ref{ass assum:L2_lip_b_sigma}. In particular, if
\(
|b(t,\cdot)-b(s,\cdot)| \le C |t-s|^{\beta}
\)
for $\beta \in \Big[\tfrac14,\tfrac12\Big],$
then by following arguments similar to those in the proof of Lemma \ref{le lemma:X_bar_minus_X}, we obtain
\(
\sup_{t\in[0,T]} \sE\!\left[\left|\cX^{0,\xi,\alpha}_t 
- \bar{\cX}^{0,{}^\pi\!\xi,\bar\alpha}_t\right|^2\right]
\le C\, h^{2\beta}.
\)
In this paper, we work under the weakest  assumption with
\(\beta = \tfrac14\), since smaller values of \(\beta\) do not lead to any
improvement in the final bound of Theorem \ref{thm:V_minus_barV}.
    \end{remark}

We are now ready to present the proof of Theorem \ref{thm:V_minus_barV}, based on the bound between $\cX$ and its Euler update $\bar \cX$ using Lemmas \ref{le lemma:X_bar_minus_X} as well as the continuity of cost functions $f, g$, and the convexity assumptions imposed in Assumption \ref{ass assump:Hstruct}.

\paragraph{Proof of Theorem \ref{thm:V_minus_barV}}
To prove \eqref{eq:V_minus_barV} , we   prove the following two inequalities.
\paragraph{Step 1} We first prove that
   $ \bar V(0,\mu) \le V(0,\mu) +Ch^{\frac14},
   $
for some constant $C > 0$.

Let $\xi\in\bS_2(\sF).$
Fix $\vep>0$ and choose $\a$ with $J\left(0,\mu,\a\right) \leq V(0,\mu)+\vep$. 
Let $\bar{\alpha} \in \cA^{\pi}$ be the discrete-time control constructed from $\alpha$ according to Remark \ref{rmk:construct_of_a}(ii). 
Let $\cX\coloneqq \mathcal{X}^{0,\xi,\alpha}$ be the solution to \eqref{eq:state_dynamics_dX}, and $\bar \cX \coloneqq \bar{\mathcal{X}}^{0,{}^\pi\xi,\bar{\alpha}}$ the Euler scheme \eqref{eq de euler scheme}.

By \eqref{eq:objective_J} and \eqref{eq objective_bar_J}, we have
\begin{align}
    J(0,\xi,\alpha) - \bar J(0, {}^\pi\!\xi,\bar \alpha)
           &= \sum_{n=0}^{N-1}  \sE\!\left[ \int_{t_n}^{t_{n+1}}
           \Delta f_n(t) ~\d t \right]
           + \sE\left[  \Terminal(\cX, \cL(\cX)) - \Terminal( \les, \cL( \les))\right],
           \label{eq terminal cost}
\end{align}
where $\Delta f_n(t) \coloneqq f \left(t, \cX_{\cdot \wedge t}, \cL(\cX_{\cdot \wedge t}), \a_t \right)
           - f(t_n,   \les_{\cdot \wedge t_n}, \cL(\les_{\cdot \wedge t_n}), \bar \a_{n})$.
Observe that, for $t \in [t_n, t_{n+1}],$
\begin{equation}\label{eq:f_minus_f}
    \begin{aligned}
   &f \left(t, \cX_{\cdot \wedge t}, \cL(\cX_{\cdot \wedge t}), \a_t \right)  - f(t_n,   \les_{\cdot \wedge t_n}, \cL(\les_{\cdot \wedge t_n}), \bar \a_n)\\
   &
   = f \left(t, \cX_{\cdot \wedge t}, \cL(\cX_{\cdot \wedge t}), \a_t \right)  - f\left(t_n, {}^\ell\!{\cX}_{\cdot \wedge t_n}, \cL({}^\ell\!\cX_{\cdot \wedge t_n}), \a_t \right)
   \\
   &\qquad + f\left(t_n, {}^\ell\!{\cX}_{\cdot \wedge t_n}, \cL({}^\ell\!\cX_{\cdot \wedge t_n}), \a_t \right) - f\left(t_n, {}^\ell\!{\cX}_{\cdot \wedge t_n}, \cL({}^\ell\!\cX_{\cdot \wedge t_n}), \bar\a_{n} \right)
\\ 
   &\qquad +f\left(t_n, {}^\ell\! {\cX}_{\cdot \wedge t_n}, \cL({}^\ell\! \cX_{\cdot \wedge t_n}), \bar\a_{n} \right)
   -f\left(t_n, {}^\ell\!\bar{\cX}_{\cdot \wedge t_n}, \cL({}^\ell\!\cX_{\cdot \wedge t_n}), \bar{\alpha}_{n} \right)
\end{aligned}
\end{equation}
First, note that by construction of $\bar \a$ in Remark \ref{rmk:construct_of_a}(ii),
\begin{equation}\label{eq:f_minus_f_part_a}
    \esp{\int_{t_n}^{t_{n+1}}f\left(t_n, {}^\ell\! {\cX}_{\cdot \wedge t_n}, \cL({}^\ell\! \cX_{\cdot \wedge t_n}), \a_t \right)
   -f\left(t_n, {}^\ell\! {\cX}_{\cdot \wedge t_n}, \cL({}^\ell\!\cX_{\cdot \wedge t_n}), \bar{\alpha}_{n} \right) \d t} \ge 0.
\end{equation}
   To analyse the other two terms in \eqref{eq:f_minus_f}, note that
by Assumption \ref{assum:terminal_cost_f_g_Phi}, we have
\begin{align}  
   & \esp{
   \left|f \left(t, \cX_{\cdot \wedge t}, \cL(\cX_{\cdot \wedge t}), \a_t \right)  - f\left(t_n, {}^\ell\! {\cX}_{\cdot \wedge t_n}, \cL({}^\ell\! \cX_{\cdot \wedge t_n}), \a_t \right)\right| 
   } \nonumber
  \\
  &\leq
 \sE \Big[ \left(1 +\|\cX_{\cdot \wedge t}\|_{\infty} + \|\upl \cX_{\cdot \wedge t_n}\|_{\infty} + \cWinf( \cL(\cX_{\cdot \wedge t}), \delta_0) + \cWinf(\cL(\upl \cX_{\cdot \wedge t_n}), \delta_0)  \right) \nonumber\\
 &\qquad \qquad 
  \cdot \left(|t-t_n|^\frac14+ \|\cX_{\cdot\wedge t_n} - {}^\ell\! \cX_{\cdot\wedge t_n} \|_{2} + \cWtwo( \cL(\cX), \cL({}^\ell\! \cX)) \right) \Big]
  \nonumber
  \\
 &\leq C \left(1+\sE \left[\|\cX\|_\infty^2\right]^{\frac{1}{2}} +\sE \left[\|{}^\ell\! \cX\|_\infty^2\right]^{\frac{1}{2}} \right) \notag \\
 &\qquad \cdot \left( h^\frac14  + \sE\left[\|\cX_{\cdot \wedge t_n} -\upl \cX _{\cdot \wedge t_n} \|_{2}^2\right]^\frac{1}{2} + \sE\left[\|\cX_{\cdot \wedge t} - \cX _{\cdot \wedge t_n} \|_{2}^2\right]^\frac{1}{2} \right)
 \nonumber
 \\
 &\overset{(a)}{\leq} C \sE\left[\|\cX _{\cdot \wedge t_n}-\upl \cX _{\cdot \wedge t_n}\|_{2}^2\right]^\frac{1}{2} + Ch^\frac14 
\overset{(b)}{\leq} C {h}^\frac{1}{4}, \label{eq long computation f}
\end{align}
where $(a)$ follows from
Propositions \ref{prop:X_square_bound} and  the fact that $\| \upl \cX\|_\infty \leq \| \cX\|_\infty $, and $(b)$ is a result of Lemma \ref{le prop interpol}.  
We similarly compute 
\begin{align}
   &\esp{
   \left|f\left(t_n, {}^\ell\!{\cX}_{\cdot \wedge t_n}, \cL({}^\ell\!\cX_{\cdot \wedge t_n}), \bar\a_{n} \right)  - f\left(t_n, {}^\ell\!\bar{\cX}_{\cdot \wedge t_n}, \cL({}^\ell\!\bar\cX_{\cdot \wedge t_n}), \bar\a_{n}\right)\right| 
   }   \notag \\
   &\le C  
    \sE \left[  \left\|  \upl\cX_{\cdot \wedge t_n}- \upl\bar \cX_{\cdot \wedge t_n} \right\|_2^2 \right]^\frac12
 \le C\sup_{s\leq T} \sE\left[\left|{}^\ell\!\cX_s - {}^\ell\! \bar\cX_s \right|^2\right]^{\frac{1}{2}} \le Ch^\frac14,
   \label{eq long computation f 2}
\end{align}
where the last inequality follows from Lemma \ref{le lemma:X_bar_minus_X}. 
By applying \eqref{eq:f_minus_f_part_a}, \eqref{eq long computation f}, and \eqref{eq long computation f 2} to \eqref{eq:f_minus_f}, we have
\begin{align}\label{eq majo f final}
    \esp{\int_{t_n}^{t_{n+1}} f \left(t, \cX_{\cdot \wedge t}, \cL(\cX_{\cdot \wedge t}), \a_t \right)  - f(t_n,   \les_{\cdot \wedge t_n}, \cL(\les_{\cdot \wedge t_n}), \bar \a _n)~ \d t}
\ge - Ch^{\frac14} h.
\end{align}

For the terminal cost term $g$ in \eqref{eq terminal cost},  we compute by Assumption \ref{assum:terminal_cost_f_g_Phi},  Proposition \ref{prop:X_square_bound}, Proposition \ref{pr prop:moment_estimate_first_euler}, and Lemma \ref{le lemma:X_bar_minus_X},
\begin{align}
   & \left| \sE\left[  \Terminal(\cX, \cL(\cX)) - \Terminal(    \les, \cL(  \les))\right] \right|  \notag \\
    &\leq C \sE \Big[ \left(1 +\|\cX\|_{\infty} + \| \les\|_{\infty} + \cWinf( \cL(\cX), \delta_0) + \cWinf(\cL( \les), \delta_0)  \right)\notag \\
    &\qquad\qquad\cdot \left( \|\cX -  \les\|_{2} + \cWtwo( \cL(\cX), \cL( \les)) \right) \Big] \leq C h^{\frac{1}{4}}. \label{eq:J_minus_barJ_2}
\end{align}
Inserting \eqref{eq majo f final} and \eqref{eq:J_minus_barJ_2} back into \eqref{eq terminal cost}, we obtain 
$     \bar J(0, {}^\pi\!\xi,\bar \alpha) \le  J(0,\xi,\alpha) + Ch^{\frac14} $
which leads to $
    \bar J(0, {}^\pi\!\xi,\bar \alpha) \le V(0,\mu) + \epsilon +Ch^{\frac14}. $
We conclude the proof of this step by letting $\epsilon \rightarrow 0$.

\paragraph{Step 2} We now show the converse inequality:
    $V(0,\mu)  \le \bar V(0,\mu)  +Ch^{\frac14}.$

Let $\bar \a \in \cA^\pi$ such that $\bar J(0,\mu,\bar\a) \leq \bar V(0,\mu) +\vep.$  
 We also denote by $\bar\alpha$ its piecewise-constant extension (observe that $\bar{\alpha} \in \cA$).
We then consider the control dynamics $\cX:=\cX^{0,\xi,\bar{\alpha}}$ and the associated controlled Euler scheme $\es:=\bar\cX^{0,{}^\pi\!\xi,\bar{\alpha}}$.
Then, using similar arguments as in \eqref{eq long computation f} and \eqref{eq:J_minus_barJ_2} using \eqref{eq:X_bar_minus_X2}, we obtain
 $  |J(0,\xi,\bar\alpha) - \bar J(0, {}^\pi\!\xi,\bar \alpha)| \le Ch^{\frac14}. $
This leads to 
\(V(0,\mu) \le  \bar J(0, {}^\pi\!\xi,\bar \alpha) + Ch^\frac14,
\)
and from the choice of $\bar\alpha$ to 
\(
V(0,\mu) \le \bar{V}(0,\mu)+\epsilon+Ch^\frac14.
\)
The proof for this step is concluded by letting $\epsilon$ goes to $0$. \eproof

\section{Feedback Controls in Paths}\label{sec:b_feedback}

In this section, we introduce two classes of feedback controls associated with the Euler scheme \eqref{eq de euler scheme}. The first class consists of feedback controls depending on the state trajectory of the Euler scheme itself. This formulation is natural from the dynamic programming point of view and, by Theorem \ref{thm:same-value-on-fdbck-ctrl}, yields the same value function as the discrete open-loop formulation \eqref{eq barV}. 
The second class consists of feedback controls depending only on the initial condition and on the Brownian increments $(\Delta W_k)_k$. 
We shall show that this class of controls also leads to the same value function in Theorem \ref{thm:same-value-on-fdbck-ctrl}.

\paragraph{Feedback controls in the state process $\bar{\cX}$}

We define a feedback control as a sequence
$\psi^X=(\psi^X_k)_{0\le k\le N-1}$, where, for each
$0\le k\le N-1$,
\(
\psi^X_k:(\R^d)^{k+1}\to A
\)
is a measurable map. The collection of such sequences is denoted by
$\cM^X$.
For $\xi\in\mathbf S_2(\mathbb F^\pi)$ and
$\psi^X\in\cM^X$, we associate the control process
$\bar\alpha^{\psi^X}\in\cA^\pi$ defined by
\begin{align}\label{eq def bar alpha X}
\bar\alpha^{\psi^X}_{k}
:=
\psi^X_k (
\bar{\mathcal X}^{t_n,\xi,X}_{\cdot\wedge t_k}
),
\qquad 0\le k\le N-1.
\end{align}
With a slight abuse of notation, we denote   
$\bar \a^X $ the control $\bar \a^{\psi^X}$, and
$\bar{\mathcal X}^{t_n,\xi,X}$
the Euler scheme
$\bar{\mathcal X}^{t_n,\xi,\bar\alpha^X}$
defined in \eqref{eq de euler scheme}. 
Note that $\bar\alpha^{\psi^X}$ in \eqref{eq def bar alpha X} is defined inductively: for each $k$, $\bar{\mathcal X}^{t_n,\xi,X}_{\cdot\wedge t_k}$ depends only on the values of $\bar\alpha^{\psi^X}$ up to step $k-1$.

\paragraph{Feedback controls in the Brownian increments $(\Delta W_k)_k$}

We next introduce a second class of feedback controls, where the control at time $t_k$ depends on the initial condition together with the Brownian increments $(\Delta W_1,\ldots,\Delta W_k)$. 
Let $\cM^W$ denote the set of sequences
$\psi^W=(\psi^W_k)_{0\le k\le N-1}$ of measurable maps
\(
\psi^W_k:\R^d\times(\R^{\bmdim})^k\to A.
\)
For any $\psi^W\in\cM^W$, we define the associated control process
$\bar\alpha^W\in\cA^\pi$ by
\begin{align}\label{eq de beta}
    \bar\alpha^{\psi^W}_{k}
    =
    \psi^W_k(\xi_0,\Delta W_1,\ldots,\Delta W_k),
    \qquad 0\le k\le N-1,
\end{align}
for any $\xi\in\bS_2(\sF^\pi)$.
With a slight abuse of notation,  we denote   
$\bar \a^W$ the control $\bar \a^{\psi^W}$. For $\psi^W\in\cM^W$, we write
\(
\bar{\cX}^{0,\xi,W}
\) or \(
\bar{\cX}^{0,\xi,\bar\alpha^W}
\)
interchangeably for the Euler scheme defined in
\eqref{eq de euler scheme}, where the control process is given by
\eqref{eq de beta}.

We then define, for any $\xi\in\mathbf S_2(\mathbb F^\pi)$,
\begin{align}\label{eq def fdbck ctrl increments}
    \widehat V_0(\xi)
    :=
    \inf_{\psi^W\in\cM^W}
    \bar J\bigl(0,\xi,\bar\alpha^W\bigr),
\end{align}
where $\bar J$ is defined in \eqref{eq objective_bar_J}.

The following proposition shows that the optimal value obtained over the two classes of feedback controls coincides with the value function $\bar V$ in \eqref{eq barV}, defined over the discrete open-loop control set $\cA^\pi$.

\begin{theorem}\label{thm:same-value-on-fdbck-ctrl}
Let $\bar V$ be defined as in \eqref{eq barV}. Then, for any
$\xi\in\bS_2(\sF^\pi)$,
\begin{align}\label{eq same value with feedback control}
    \bar V(0,\xi)
    =
    \inf_{\psi^X\in\cM^X}
    \bar J(0,\xi,\bar\alpha^X)
    =
    \inf_{\psi^W\in\cM^W}
    \bar J(0,\xi,\bar\alpha^W).
\end{align}
\end{theorem} 
\proof 
We denote by $v(t_n,{\xi}) =\inf_{\psi^X \in \cM^X} \bar J(t_n, \xi, \bar \a^X) $. We first show that $\bar V = v$ in Step 1 and 2. Then we will show the equivalence between $v$ and $\widehat V$ in Step 3.

\paragraph{Step 1: $\bar V \leq v$, $ \bar V(0,\xi) \leq \widehat V_0(\xi)$} 
This step is immediate from the definition of $\bar V$, $v$, and $\widehat V_0$.

\paragraph{Step 2: $\bar V \geq  v$}  
Fix \( 1 \leq n \leq N \), \( \xi \in \mathbf{S}_2(\sF^\pi)\)  and denote the probability distribution $\mu_n \coloneqq\cL(\upl \xi_{\cdot\wedge t_n})$. 
For every $t_n \in \pi$, we also denote the set
\begin{equation}\label{eq:M_n}
    \cM_{t_n} \coloneqq \{ \alpha: \Omega ~\ra A: \alpha \text{ is } \cF_{t_n}\text{-measurable} \}.
\end{equation}
Define the function $F(\cdot ; \mu_n): \sR^{dN} \times A \mapsto \sR$ such that 
\begin{equation} \label{eq:F_mu}
    F(x, a; \mu_n) := f\left(t_n, \upl x, \mu_n, a\right) h + \bar V \left( t_{n+1}, \bar \cX^{t_n, x, \mu_n, a} \right),
\end{equation}
where $\bar \cX^{t_n, x, \mu_n, a}$ denotes the Euler scheme \eqref{eq de euler scheme} from a fixed initial state $x = (x_i)_{i=0}^N \in \sR^{d N}$ with a fixed measure $\mu_n \in \cP_2(\sR^{dN})$, 
i.e., 
$$
\bar{\mathcal{X}}_{t_{n+1}}^{t_n, x, \mu_n, a} = x_n +b\left(t_n, \upl x , \mu_n, a \right) h+\sigma\left(t_n, \upl x , \mu_n \right)\Delta W_{\textcolor{jfc}{n}} .
$$
By DPP in Lemma~\ref{lemma:DPP},    \begin{equation}\label{eq:VbarDPP_update1}
        \begin{aligned}
            \bar V(t_n,\xi) &= \inf_{\bar \alpha \in\cA^\pi}\left\{\mathbb{E}\left[ f\left(t_n, \upl \bar \cX^{t_n, \xi, \bar \alpha}, \cL \left({\upl \bar \cX_{\cdot \wedge t_n}^{{t_n}, \xi, \bar \alpha}}\right),\bar \alpha_{n}\right) h\right]+ \bar V\left({t_{n+1}}, \bar \cX^{{t_{n}}, \xi, \bar \alpha}\right)\right\} \\
            &= \inf_{\bar \alpha \in\cA^\pi}\left\{\mathbb{E}\left[ f\left(t_n, \upl \xi_{\cdot \wedge t_n}, \cL\left({ \upl \xi_{\cdot \wedge t_n}}\right),\bar \alpha_{n}\right) h\right]+ \bar V\left({t_{n+1}}, \bar \cX^{{t_{n}}, \xi, \bar \alpha}\right)\right\}\\
            &= \inf_{\bar \alpha \in \cA^\pi} \sE [F(\xi_{\cdot \wedge t_n}, \bar \alpha_n; \mu_n)] ,
        \end{aligned}
    \end{equation}
    where the second equality follows from the non-anticipating property of $f$ (cf. Remark \ref{rmk:nonanticipate_fg}) and the last equality follows from \eqref{eq:F_mu}. 

We next present a measurable selection theorem adapted from \cite[Lemma F.3]{cosso2023optimal}. Since the proof follows directly by applying the same arguments as those in \cite[Lemma F.3]{cosso2023optimal} to discretised processes in $\bS_2(\sF^\pi)$, we omit it here:
\begin{lemma}[Measurable selection]\label{lemma:measurable_selection}
    Fix $t_n \in \pi$, $\xi \in \mathbf{S}_2({\sF^\pi})$ with probability distribution $\mu \in \cP_2(\sR^{dN})$. 
    Let $F: \sR^{dN} \times A \ra \sR$ be a measurable function. Denote $\cM$ as the set of Borel-measurable maps  $\{a: \sR^{dN}\ra A \}.$
    Suppose that $\mathbb{E}[|F(\xi, \a)|] <
+\infty,$ for any $\a \in \mathcal{M}_{t_n}$ with $\cM_{t_n}$ defined in \eqref{eq:M_n}, and also that $\inf _{a \in A} F(x, a)>-\infty, $ for any $x \in \sR^{dN}$. Then,
    \begin{equation}\label{eq:mble_equality}
        \inf_{\a \in \mathcal{M}_{t_n}} \mathbb{E}[F(\xi, \a)]=\inf_{\a \in \mathcal{M}} \mathbb{E}[F(\xi, \a(\xi))] = \sE \left[ \textrm{ess} \inf_{a \in A} F(\xi, a)\right].
    \end{equation}
Moreover, for any $\vep > 0$, there exists a Borel measurable function $\hat\alpha_{\mu}^\vep:\sR^{dN} \ra A$ such that 
$$
 \sE\left[F(\xi, \hat \alpha_\mu(\xi))\right] \leq \sE \left[ \text{ess} \inf_{a \in A} F(\xi, a) \right] +\vep.
$$
\end{lemma}

    Fix $\vep > 0$. According to Lemma \ref{lemma:measurable_selection},  there exists a measurable function $\hat \alpha_{\mu_n} (\cdot)$ such that
$$
\begin{aligned}
    \inf _{\bar \alpha \in \cA^\pi} \sE [F( \xi_{\cdot \wedge t_n}, \bar \alpha_n; \mu_n)]& \geq \sE \left[ F( \xi_{\cdot \wedge t_n}, \hat\alpha_{\mu_n}( \xi_{\cdot \wedge t_n}); \mu_n) \right] - \frac{\vep}{N} \\
    &= \sE \left[ f(t_n, \upl\xi, \mu_n, \hat\a_{\mu_n}( \xi_{\cdot \wedge t_n})) h   \right ] + \bar V\left(t_{n+1}, \bar \cX^{t_n, \xi,  \hat \alpha} \right) -\frac{\vep}{N} .
\end{aligned}
$$
Then combining the above inequality  with \eqref{eq:VbarDPP_update1}, we obtain
\begin{equation}\label{eq:mble_V}
\begin{aligned}
 \bar V(t_n, \xi) 
      &\geq  \sE \left[ f(t_n, \upl \xi, \mu_n, \hat\a_{\mu_n}(\xi_{\cdot \wedge t_n})) \right ] h  + \bar V\left(t_{n+1}, \bar \cX^{t_n, \xi, \hat \alpha} \right) - \frac{\vep}{N}.
\end{aligned}
\end{equation}

Now define a control sequence \( \left\{ \hat \alpha_{\mu_n}\left( \bar \cX_{\cdot \wedge t_n}^{0, \xi, \hat{\a}}  \right) \right\}_{1\leq n\leq N} \) recursively, with $\mu_k = \cL(\upl \bar \cX_{\cdot \wedge t_k}^{0, \xi, \hat{\a}})$ for any $k \in [N]$.
By Proposition \ref{pr prop:moment_estimate_first_euler}, $\bar \cX^{0,\xi, \hat{\a}} \in \mathbf{S}_2(\sF^\pi).$ Then by applying \eqref{eq:mble_V} and the flow property (cf. Proposition \ref{pr prop:moment_estimate_first_euler}(ii)) recursively,
we have
$$
\begin{aligned}
      &\bar V(0,\xi)\geq \sum_{n=0}^{N-1} \sE\!\left[f\!\left( t_n,  \upl \bar \cX_{\cdot \wedge t_n}^{0,\xi,\hat \a},
        \cL\!\left( \upl \bar \cX_{\cdot \wedge t_n}^{0,\xi,\hat \a}\right),
        \hat\a_{\mu_n}\!\left( \upl \bar \cX_{\cdot \wedge t_n}^{0,\xi,\hat \a}\right) \right)\right] h
        +\bar V(T, \bar \cX^{0,\xi, \hat{\a}}) -\vep  \\
     &= \bar J(0,\xi, \hat{a}) - \vep.
\end{aligned}
$$
Therefore, $\hat{\a}$ is an $\vep$-optimal solution to \eqref{eq barV}. By taking $\vep \downarrow 0$, we finish the proof of $\bar V \geq   v.$ 

\paragraph{Step 3: $v= \widehat V$}

Note with $\hat\a \in \cM^X$ and \eqref{eq de euler scheme}, $\bar{\mathcal X}_{t_i}^{0,\xi,\hat{a}}$ is a deterministic function of $\xi_0$ and $(\Delta W_k)_{0\le k\le i}$.
Hence, there exists a function $\psi^W \in \cM^W$ such that $\bar J(0,\xi, \bar\alpha^W) = \bar J(0,\xi, \bar\alpha^{\hat\a})$, which is an $\vep$-optimal solution to \eqref{eq barV}. By taking $\vep \downarrow 0$, it shows $ \widehat{V} \geq v$.  With Step I, the equivalence between $v$ and $\widehat V$ is established.
\eproof

\begin{remark}\label{rmk:law_invariance_b}
Using Remark \ref{re intrinsic value disc time control}, we naturally define the intrinsic value function associated with the discrete-time control problem under feedback controls driven by the Brownian increments. For $\mu\in\cP_2((\R^d)^N)$, we also write with slight abuse of notations,
\(
\widehat V_0(\mu):=\widehat V_0(\xi),
\) whenever $\cL(\xi) = \mu.$
  
\end{remark}

The following corollary provides an estimate for the difference between the value function $\widehat V$ defined in \eqref{eq def fdbck ctrl increments}, corresponding to feedback controls depending on the initial condition and the Brownian increments, and the value function $V$ associated with the continuous-time MKV dynamics \eqref{eq:state_dynamics_dX}. This result follows directly from Theorem \ref{thm:same-value-on-fdbck-ctrl} together with Theorem \ref{thm:V_minus_barV}.

\begin{corollary} \label{thm:hat_v_minus_v}
Suppose that Assumptions  \ref{ass assum:L2_lip_b_sigma}, \ref{assum:terminal_cost_f_g_Phi}  and  \ref{ass assump:Hstruct} hold. For $\mu \in \cP_{\infty,2}$, 
    $$
  |  V(0,\mu) - \widehat{V}_0(\mathfrak{p}^\pi\sharp\mu) | \leq C h^\frac{1}{4},
    $$
    where $C$ is a positive constant that does not depend on $\pi$.
\end{corollary}

\section{Fully Implementable Schemes and Numerical Approximation}\label{sec:fully_implement_scheme}

\subsection{Particle system approximation}\label{sec:particle_approx}

We now introduce a particle approximation of the value function $\widehat V_0$ defined in \eqref{eq def fdbck ctrl increments}. This approximation provides the basis for the numerical implementation of the optimization problem \eqref{eq:vM}. Since the numerical scheme developed below is based on feedback controls depending on Brownian increments and the initial condition, we focus throughout this section on the class $\cM^W$ introduced in Section \ref{sec:b_feedback}.

Let $M\in\mathbb{N}^+$ denote the number of particles. Let $(W^m)_{1\le m\le M}$ be a collection of independent Brownian motions, and let $(\xi^m)_{1\le m\le M}$ be i.i.d. copies of the initial condition $\xi\sim\mu$, where $\mu$ is a prescribed initial distribution in $\cP_2((\R^d)^N)$. For each $m \in [M]$ and a given Brownian-increment feedback control $\psi^W\in\cM^W$, define the particle controls by
\begin{equation}\label{eq:bar_alpha_W_particle}
    \bar\alpha^{W,m}_0
    \coloneqq
    \psi^W_0\bigl(\xi^m_0\bigr),
    \qquad
    \bar\alpha^{W,m}_n
    \coloneqq
    \psi^W_n\bigl(\xi^m_0,\Delta W^m_1,\ldots,\Delta W^m_n\bigr),
    \qquad 1\le n\le N-1 ,
\end{equation}
in parallel with \eqref{eq de beta}. The corresponding interacting particle system is given by
\begin{equation}
    \begin{aligned}
    X^m_{t_0} &= \xi^m_0, \\
    X^m_{t_{n+1}} 
    &= X^m_{t_n}
    + b_n^m\, h + \sigma_n^m\,\Delta W^m_{n+1},
    \qquad 0\le n\le N-1,
    \end{aligned}
    \label{eq:ParticleSystem}
\end{equation}
with $b_n^m \coloneqq b(t_n,{}^\ell X^m_{\cdot\wedge t_n},\mu^M_{t_n},\bar\alpha^{W,m}_n)$,
$\sigma_n^m \coloneqq \sigma(t_n,{}^\ell X^m_{\cdot\wedge t_n},\mu^M_{t_n})$, and, for $t\in\pi$,
$\mu^M_t \coloneqq \frac{1}{M}\sum_{k=1}^M \delta_{{}^\ell X^k_{\cdot\wedge t}}$.

Under our standing assumptions, one readily verifies the uniform moment estimate
\begin{equation}\label{eq:moment_bound_particle_approx}
    \mathbb{E}\left[
    \max_{0\le n\le N}\left|X^m_{t_n}\right|^2
    \right]
    \le
    C\left(
    1+\mathbb{E}\left[\|\xi^m_{\cdot\wedge 0}\|_\infty^2\right]
    \right),
    \qquad 1\le m\le M .
\end{equation}

Similar particle approximation can also be formulated for state-path feedback controls $\psi^X\in\cM^X$. In particular, one may define
\begin{equation}\label{eq:bar_alpha_X_particle}
    \bar\alpha^{X,m}_n
    \coloneqq
    \psi^X_n\bigl(X^m_{t_0},\ldots,X^m_{t_n}\bigr),
    \qquad 0\le n\le N-1 .
\end{equation}
Replacing $\bar\alpha^{W,m}_n$ by $\bar\alpha^{X,m}_n$ in \eqref{eq:ParticleSystem} yields the corresponding particle system. In the following section, we focus on the particle approximation result with $\psi^W$. 

Let $J^M$ denote the objective function for the particle system in \eqref{eq:ParticleSystem}: for any  $\mu \in \cP_{2}((\R^d)^N)$, and $\psi^W \in \cM^W$,
\begin{equation}\label{eq:objective_J_M}
     J^M(0, \mu, \psi^W) \coloneqq \mathbb{E} \left[\frac{1}{M} \sum_{m=1}^M\sum_{n=0}^{N-1} f\left(t_n, \upl X^m_{\cdot \wedge t_n},  \mu_{t_n}^M, \bar\alpha^{W,m}_n\right) h + \Terminal\left(  \upl X^m,  \mu_{T}^M\right) \right],
\end{equation}
Finally, we introduce
\begin{equation}\label{eq:vM}
    {V}^M_0(\mu)=\inf _{\psi^W \in \cM^W } J^M(0,\mu,\psi^W).
\end{equation}

The following theorem establishes a particle approximation error bound between the value function $\widehat V_0$ in \eqref{eq def fdbck ctrl increments}, based on the Euler scheme with the exact path law, and $V^M_0$ in \eqref{eq:vM}, based on the Euler scheme with the empirical distribution in the update. The bound is given in terms of the number of particles $M$.
\begin{theorem}\label{thm:particle_approx}   Fix $q > 2.$ 
    Let $\mu \in \cP_{q}((\R^d)^N)$. Then there exists a constant $C \geq 0$ independent of $h$ and $M$ such that
$$
\left|\widehat{V}_0(\mu)-{V}^M_0(\mu)\right| \le  C M^{-\gamma({q, d})/2}.
$$
Here $\gamma(q, d) = \min \{2 / ((N+1)d) , (q-2) / q\}$.
\end{theorem}

To prove Theorem~\ref{thm:particle_approx}, we first show that the difference between the particle paths $(X^m)$ governed by the particle system~\eqref{eq:ParticleSystem} and the process $\bar\cX$ defined in \eqref{eq de euler scheme} with control process given by \eqref{eq de beta} can be bounded in terms of the Wasserstein distance between the law of $\bar\cX$ and its empirical distribution based on $M$ particles (see Lemma~\ref{lemma:delta_X_B_Sigma}).  Then, by viewing the linear interpolated path as (finite) high dimensional vector, 
we  apply a classical result from \cite[Theorem~1]{fournier2015rate} to bound this Wasserstein distance in terms of the number of particles $M$, as stated in Lemma~\ref{lemma:wasserstein_bound}.
A complete proof of Lemma~\ref{lemma:delta_X_B_Sigma} and Theorem~\ref{thm:particle_approx} is deferred to Appendix~\ref{sssec:proof_of_thm_particle_approx}.

\begin{lemma}\label{lemma:delta_X_B_Sigma}
For $1 \leq m\leq M$, let $\bar{\cX}^m \coloneqq \bar\cX^{\xi^m,W^m}$ denote the Euler scheme defined in
\eqref{eq de euler scheme} where $W$ is replaced by $W^m$ and the control process is given by
\eqref{eq de beta} with $\bar\alpha^{W,m}_n = \psi^W_n(\xi^m_0, (\Delta W^m_k)_{1\leq k \leq n})$ for any $0\leq n \leq N-1$.
We also denote by $X^m = X^{\xi^m,W^m}$ following the particle system in \eqref{eq:ParticleSystem} with controls \eqref{eq:bar_alpha_W_particle}.
Let  $\bar\mu^M$ denote the empirical distribution of $(\bar\cX^m)_{1\leq m \leq M}$.

Then for any  $ t_i \in \pi$ and $ 1\leq m\leq M$, we have that
$$
\mathbb{E}\left[\left| \bar\cX_{t_i}^m -  X^m_{t_i}\right|^2\right] \leq C \max _{0\leq n \leq N}  \cW_{2, L^2} \left(\cL\left(\bar\cX^m_{\cdot \wedge t_n}\right), \bar\mu_{t_n}^M \right). 
$$
\end{lemma}

We then include a result about the Wasserstein distance between the true distribution and the empirical distribution from \cite[Theorem 1]{fournier2015rate} and adapt it to our case:
\begin{lemma}\label{lemma:wasserstein_bound}
For $q > 2$ and
    $\tilde d \geq 5,$  let $\mu \in \mathcal{P}_q(\mathbb{R}^{\tilde d})$ such that $\int_{\sR^{\tilde d}} |x|^q \mu (\d x) < c$ for a constant $c$. 
Let $\mu^M$ denote the empirical measure associated with $M$ i.i.d. samples with distribution $\mu$.
    There exists a constant $C$ such that, for all $M \geq 1$,
    $ \mathbb{E}\left[\mathcal{W}_2\left(\mu^M, \mu\right)^2\right] \leq C M^{-\gamma},$ where $\gamma \coloneqq \min \{2 / \tilde d , (q-2) / q\}$.
\end{lemma}

Finally, as a direct consequence of Theorems~\ref{thm:hat_v_minus_v} and \ref{thm:particle_approx}, we obtain the following bound for the difference between the value function $V$ of the continuous-time path-dependent McKean--Vlasov control problem \eqref{eq:state_dynamics_dX} and its $M$-particle approximation $V_0^M$ with feedback controls based on Brownian increments \eqref{eq:ParticleSystem}.
\begin{theorem}\label{thm:final_thm}
    Suppose that Assumptions \ref{ass assum:L2_lip_b_sigma}, \ref{assum:terminal_cost_f_g_Phi} and \ref{ass assump:Hstruct} hold, and  $(N+1)d \geq 5$. 
    Fix $q > 2.$ Suppose that for any $m\in [M], \, \sE[|\xi_0^m|^q] \leq c$ for some constant $c$. Then there exists a constant $C \geq 0$ independent of $h$ and $M$ such that for any $\mu \in \cP_{\infty, 2}$, 
    $$
  | V_0^M (0, \mathfrak{p}^\pi\sharp\mu) -V(0,\mu) | \leq C h^\frac{1}{4} + {M^{- \min \left\{ \frac{1}{(N+1)d}, \frac{q-2}{2q}\right\}}}.
$$
\end{theorem}
\subsection{Neural network approximation}

We parametrise the Brownian-increment policy $\psi^W \in \mathcal{M}^W$ \eqref{eq:bar_alpha_W_particle} and state-trajectory policy $\psi^X \in \mathcal{M}^X$ \eqref{eq:bar_alpha_X_particle}  using a residual network with weights $\theta$.
At each time step \(t_n\), the feedback policies use only information available up to \(t_n\). Thus, for particle \(m\in[M]\), the feature vectors for Brownian-increment feedback and state-trajectory feedback are given respectively by
\(
x^{W,m}_n \coloneqq (t_n,\xi^m_0,(\Delta W^m_k)_{1\le k\le n}),
\)\footnote{We note that \(X_{t_n}^m\) is already determined by \(\bigl(\xi_0^m,(\Delta W_k^m)_{1\le k\le n}\bigr)\), so including \(X_{t_n}^m\) in the input does not add new information. For implementation purposes, however, we allow an explicit dependence on \(X_{t_n}\) by using the input vector \(x_n^{W,m} = \bigl(t_n,(\Delta W_k)_{k\le n},X_{t_n}\bigr)\).} and \(
x^{X,m}_n \coloneqq (t_n,X^m_{t_0},\ldots,X^m_{t_n}).
\)
 At each time step \(n\in[N]\), the network maps the feature vector \(x_n^m\in\mathbb{R}^{d_0}\) to the control \(\psi_\theta(x_n^m)\). To minimize the particle objective \(J^M(0,\mu,\psi_\theta)\), we simulate the system forward using an Euler scheme, sampling i.i.d.\ initial states and Brownian increments at each iteration. The empirical loss is averaged across particles, and \(\theta\) is updated via Adam using gradients from automatic differentiation. The algorithm is given in Algorithm~\ref{alg:policy_gradient}.

\begin{algorithm}[h]
\begin{algorithmic}
\caption{Policy gradient training for path-dependent mean-field control}
\label{alg:policy_gradient}
\Require initial law \(\mu\), grid \(\pi\), batch size \(B\), particles \(M\), learning rate \(\eta\), iterations \(K\)
\State Initialize \(\theta\) for the neural network  $\psi_\theta$
\For{\(k=1,\dots,K\)}
    \State Sample \(B\) batches of \(M\) particles with initial states \(\xi_0^m\sim\mu\) and Brownian increments \(\Delta W_n^m\)
    \State Simulate the particle system on \(\pi\) using controls
    \(
        \bar\alpha_n^m=\psi_{\theta,n}(x_n^m),
    \)
    where \(x_n^m\) is formed from the particle history
    \State Compute
    \(
    L(\theta)
    =
    \frac{1}{B}\sum_{j=1}^B
    \frac{1}{M}\sum_{m=1}^M
    \left(
    \sum_{n=0}^{N-1} f(\cdot)h+g(\cdot)
    \right)
    \)
    \State Update \(\theta \gets \theta-\eta\nabla_\theta L(\theta)\)
\EndFor
\State \Return \(\theta\)
\end{algorithmic}
\end{algorithm}

\section{Numerical Experiments}\label{sec:numerics}

In this section, we illustrate the proposed method on a path-dependent McKean--Vlasov control problem. 

\newcommand{\eps}{\epsilon}

\subsection{Linear quadratic path-dependent MKV problem}
We consider a 
linear quadratic system governed by the following state dynamics:
\begin{equation}\label{eq:dX_in_numerics}
  \d X_t = (a X_t +c \sE[X_t] +  \alpha_t )\, \d t + \sigma \, \d W_t,
  \qquad
  X_0 \sim \mathrm{Unif}[x_0-\eps_0, x_0 + \eps_0]
\end{equation}
for some $\eps_0\geq0$, with $X_t \in \sR^d$ and $W_t \in \sR^{\bmdim}.$

The goal of the control problem is to drive the path-averaged state
\(
Y_t = \int_0^t e^{-\lambda(t-s)} X_s \, \d s,
\)
toward a prescribed target level $K>0$ at some terminal time $T$. The corresponding objective functional is defined as
\begin{equation}\label{eq:objective_AT}
    J(\alpha)
    = \mathbb{E}\!\left[
        \int_0^T{ \frac{\eta}{2}}\, \alpha_t^2 \, \mathrm{d} t   \right]
        + \frac{1}{2} \mathbb{E}\bigg[ \lvert Y_T - K \rvert ^2\bigg]
        + \frac{\gamma}{2} \mathrm{Var}\left( Y_T \right).
\end{equation}
In Appendix \ref{appendix:LQ}, we present its theoretical solution via the associated HJB equations and the Riccati ODEs. 

\paragraph{Order of Convergence}
In Figure~\ref{fig:lq_nn_combined_results}, we compare the numerical results obtained from the control
$\psi_{\theta}^W\bigl(t_n, (\Delta W_k)_{k\leq n}, X_{t_n}\bigr)$ with the reference solution computed by numerically solving the Riccati ODE.

In Figure~\ref{fig:val_loss_E5}, we first report the validation loss during training for different time discretizations with $N=4$, $8$, $16$, and $32$. The theoretical benchmark, given by the optimal solution of the Riccati equation, is shown in purple. As $N$ increases, the numerical results converge toward the optimal solution.  Figure~\ref{fig:ord_of_cvgn} estimate the convergence order. The results indicate a rate around $1$, faster than $1/4$ as given in Theorem~\ref{thm:final_thm}, which is mainly due to the additional regularity and structure of the linear-quadratic problem.

In Figure~\ref{fig:lq_vs_nn_traj}, we compare the mean trajectories of the $X$ and $Y$ processes with time step $N=32$, using both the optimal Riccati solution and Algorithm~\ref{alg:policy_gradient} with the control parametrization $\psi_\theta^W$. The proposed algorithm performs very close to the theoretical benchmark: the mean of $Y_T$ is close to the target value $K$, and the variance is comparable to that of the optimal solution.

\begin{figure}
    \centering
    \begin{subfigure}{0.32\linewidth}
        \centering
        \includegraphics[width=\linewidth]{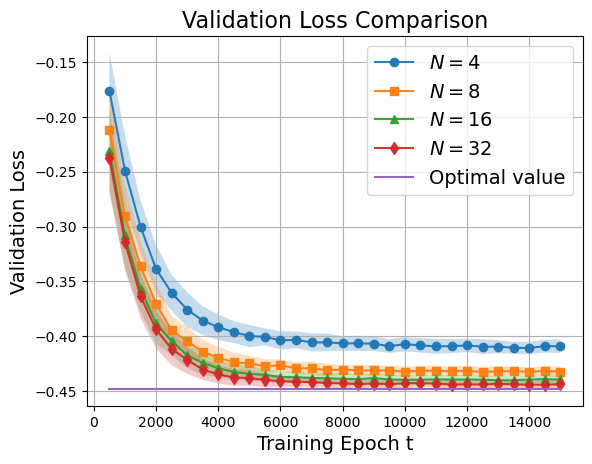}
        \caption{Validation loss  under the control $\psi_\theta^W$ with different time discretization step $h=T/N$.}
        \label{fig:val_loss_E5}
    \end{subfigure}
    \hfill    
    \begin{subfigure}{0.32\linewidth}
        \centering
        \includegraphics[width=\linewidth]{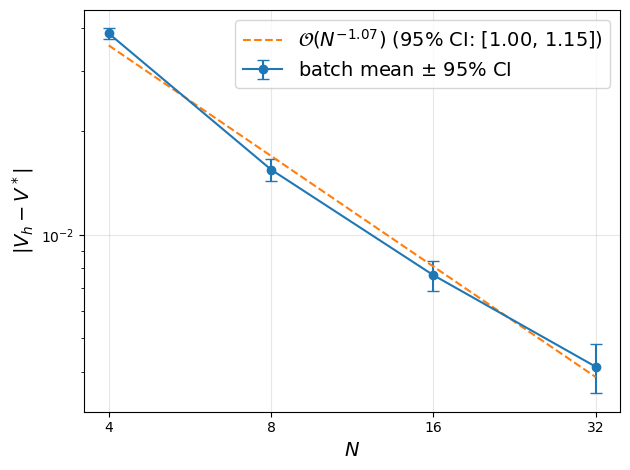}
        \caption{Order of convergence of $V_h - V^*$ under the control $\psi_\theta^W$, with time discretization step $h=T/N$.}
        \label{fig:ord_of_cvgn}
    \end{subfigure}
    \hfill
    \begin{subfigure}{0.32\linewidth}
        \centering
        \includegraphics[width=\linewidth]{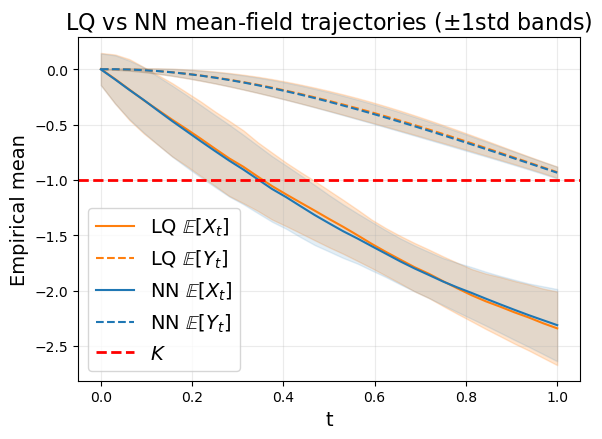}
        \caption{Mean and 1 std of the samples $X$ and $Y$ trajectories under the theoretical LQ  solution (orange) and $\psi_\theta^W$ (blue).}
        \label{fig:lq_vs_nn_traj}
    \end{subfigure}
\caption{Comparison between the theoretical solution and Algorithm~\ref{alg:policy_gradient} with control parametrization $\psi_\theta^W$  ($T=1$, $d=1$, $d^w=1$, $a=0.6$, $c=0$ , $\sigma=1$, $x_0=0$, $\epsilon_0=0.25$, $K=-1$, $\lambda=1$, $\eta=0.02$, $\gamma=2;$ for (c), $N=32$).}
    \label{fig:lq_nn_combined_results}
\end{figure}

\begin{figure}
    \centering
    \includegraphics[width=0.9\linewidth]{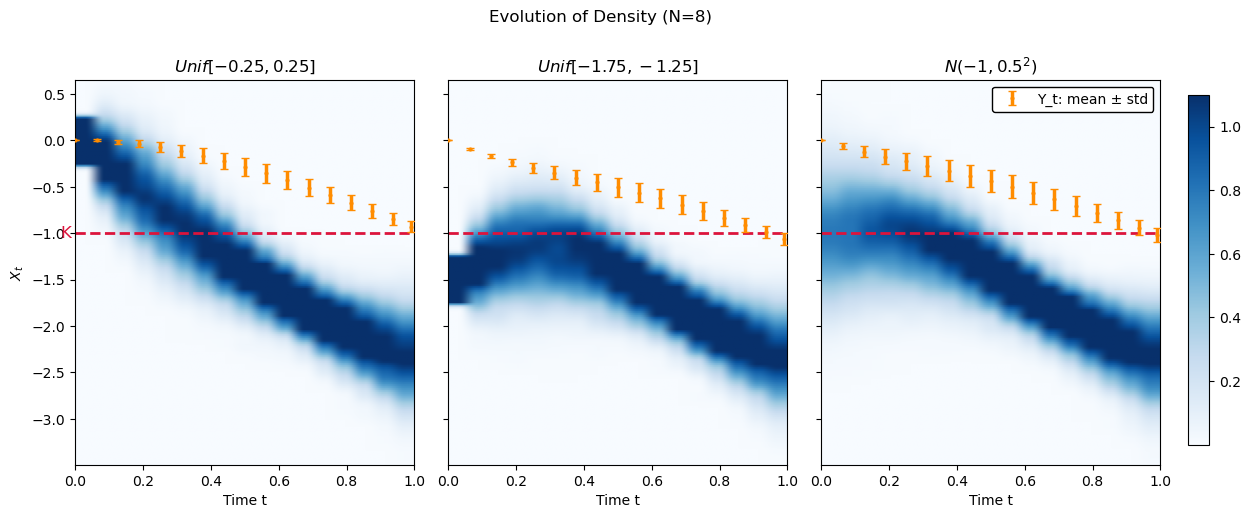}
    \centering
    \caption{
    Generalization test under three different initial distributions using the same trained neural network $\psi_\theta^W$. The evolution of $X$ is shown in blue, while the orange markers indicate the mean and standard deviation of $Y_{t_k}$. The hyperparameter settings are the same as in Figure \ref{fig:lq_vs_nn_traj}, with $N=8$.
    }
    \label{fig:gen_test_lq}
\end{figure}

\paragraph{Generalization test}
We next evaluate the performance of $\psi_\theta^W$ under different initial distributions. In Figure~\ref{fig:gen_test_lq}, the left panel uses the same initial distribution as in training, $\mathrm{Unif}[-0.25,0.25]$. The middle panel shifts the initial distribution downward by $1.5$, so that its support lies below the target value $K=-1$. The right panel replaces the uniform distribution with a normal distribution centered at the target value $K=-1$. As shown in all three cases, although $\psi_\theta^W$ is trained only with the initial distribution, it successfully drives $Y_T$ close to the target value $K$ and maintains a small terminal variance, demonstrating the generalization capability of the neural-network-based control in the path-dependent mean field setting. 

\subsection{Comparison with other control inputs in a non-smooth landscape}\label{ssec:numeric_compare}
We compare $\psi_\theta^W$ with two other types of controls. The first is $\psi_\theta^X$, defined in \eqref{eq:bar_alpha_X_particle}, which takes as input the history of state trajectory $X$. We also consider a Markovian control that depends only on the current state at time $t$, denoted by $\psi^{\text{Markov}}_\theta(t_n, X_{t_n})$. In Figure \ref{fig:loss_3compare}, we use a different loss function from \eqref{eq:objective_AT}, replacing the quadratic terminal term $|Y_T-K|^2$ with the absolute value $|Y_T-K|$. Since this term is small, the absolute-value loss highlights the differences between the control schemes more clearly and also tests the performance of the algorithm beyond the linear-quadratic setting, in a non-smooth optimization problem. Figure \ref{fig:loss_3compare} shows that $\psi_\theta^W$ and $\psi_\theta^X$ achieve the same loss value, as proved in Theorem \ref{thm:same-value-on-fdbck-ctrl}. In contrast, the Markovian control $\psi_\theta^{\text{Markov}}$ does not reach the same loss level, reflecting the more limited information available to it.
\begin{figure}
    \centering
    \includegraphics[width=0.35\linewidth]{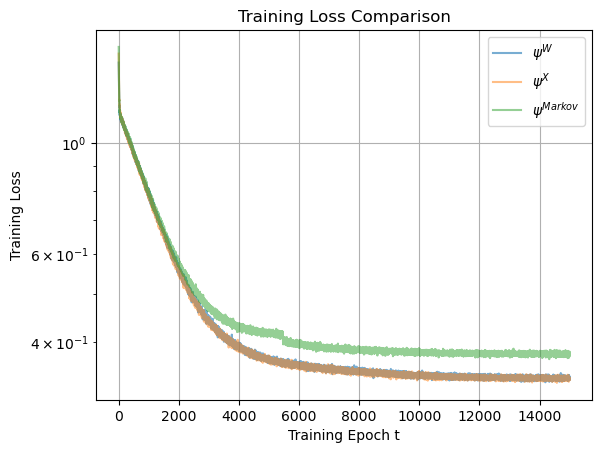}
    \includegraphics[width=0.35\linewidth]{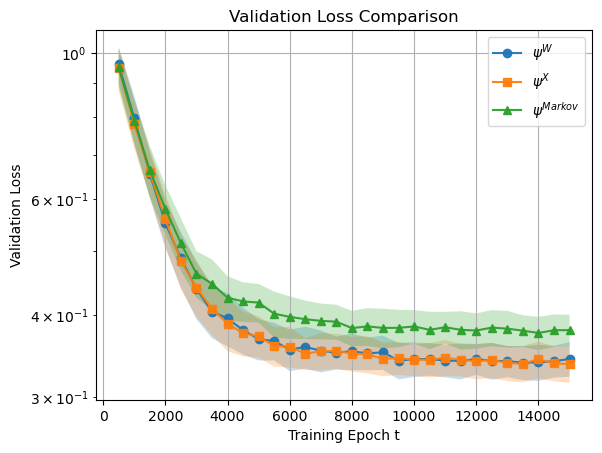}
    \caption{Training loss (left) and validation loss (right) of three controls: $\psi^W_\theta$, $\psi^X_\theta$, and $\psi^{\text{Markov}}_\theta$ ($T=1$, $d=1$, $d^w=1$, $a=0.6$, $c=-1.2$, $\sigma=1$, $x_0=0$, $\epsilon_0=0.1$, $\lambda=0$, $\eta=0.02$, $\gamma=1$; $N=4$).}
    \label{fig:loss_3compare}
\end{figure}

\section{Detailed Proofs}\label{sec:proofs}

\subsection{Proofs in Section \ref{sec:continuous_sde}}

\subsubsection{Proof of Proposition \ref{prop:X_square_bound}} \label{ssec:prop:X_square_bound}
The proof of $(i)$ and $(ii)$ follows directly from \cite{cosso2023optimal}.

We now prove part~$(iii)$.
     Fix $ 0 \leq h' \leq h$ and $u \in [t,T],$ and w.l.o.g assume that $u + h' \leq T.$  
     To ease the notation, we denote $\cX=\cX^{t, \xi,\a}$, $b_s = b(s, \cX_{\cdot \wedge s}, \cL(\cX_{\cdot \wedge s}), \a_s)$ and $\sigma_s  = \sigma(s, \cX_{\cdot \wedge s}, \cL(\cX_{\cdot \wedge s})). $
     For any $h' \in [0,h]$, 
     $
     \cX_{u + h'} - \cX_u = \int_u^{u+h'} b_s \d s + \int_u^{u+h'} \sigma_s \d W_s,
     $
   which implies that
     \begin{align}\label{eq:X_minus_X_q}
         \sE\left[ \sup_{h'\in[0,h]} |\cX_{u+h'} - \cX_u|^q\mid \cF_u\right]
     &    \leq 
         2^{q-1} \Big(
         \sE\left[\sup_{h'\in [0,h]}\left|\int_{u}^{u + h'} b_s \d s\right|^q \mid\cF_u\right] \notag \\
         &\qquad +   \sE\left[\sup_{h'\in [0,h]}\left| \int_{u}^{u+h'} \sigma_s \d W_s\right|^q\mid\cF_u\right] \Big).
     \end{align}
     By H\"older's inequality, $\left|\int_{u}^{u+h'} b_s \d s\right| ^q \leq h^{q-1} \int_u^{u+h}|b_s|^q \d s$,  hence, 
    \begin{align}\label{eq:int_bs}
       \sE\left[ \sup_{h'\in [0,h]}\left|\int_{u}^{u+h'} b_s \d s\right|^q \mid \cF_u\right] \leq h^{q-1} 
       \sE \left[ \int_u^{u+h} |b_s|^q \d s \mid \cF_u \right] \leq h^{q-1}  \int_u^{u+h} \sE \left[ |b_s|^q \mid \cF_u \right] \d s .
     \end{align}
     By Assumption \ref{ass assum:L2_lip_b_sigma} and \eqref{eq:wasserstein},
     \begin{equation}\label{eq:bs_q_bound}
              \begin{aligned}
         \sE \left[ |b_s|^q \mid \cF_u \right] & \leq C \sE\left[1+\|\cX\|_{L^\infty(s)}^q + \cWinf (\cL(\cX_{\cdot \wedge s}),\delta_0)^q\mid\cF_u\right] \\
         &\leq C \left( 1+ \sE \left[ \|\cX_{\cdot \wedge s}\|_\infty^q\mid\cF_u\right] + \sE[\|\cX_{\cdot \wedge s}\|_\infty^q] \right).
     \end{aligned}
     \end{equation}
     Moreover, by the BDG inequality (cf.\ \cite[Theorem 1.111]{fabbri2017stochastic}),
     \begin{equation}\label{eq:int_sigma}
             \begin{aligned}
        &\sE\left[\sup_{h'\in [0,h]} \left| \int_{u}^{u+h'} \sigma_s ~\d W_s\right|^q \mid \cF_u\right] 
         \leq h^{\frac{q}{2}-1} \sE \left[\int_u^{u+h} \| \sigma_s \|_F^{q}~\d s \mid\cF_u\right] \\
        & \leq C h^{\frac{q}{2}-1}  \int_u^{u+h} 1+ \sE \left[ \|\cX_{\cdot \wedge s}\|_\infty^q\mid\cF_u\right] + \sE[\|\cX_{\cdot \wedge s}\|_\infty^q] ~\d s .
     \end{aligned}
     \end{equation}
Therefore, by applying \eqref{eq:int_bs}, \eqref{eq:bs_q_bound}, and \eqref{eq:int_sigma} to \eqref{eq:X_minus_X_q} and taking expectation, we obtain
     \begin{align*}
   \sE\left[  \sup_{h' \in [0, h]}|\cX_{u+h'} - \cX_u|^q \right]
         &\leq   C h^{\frac{q}{2}-1} \int_u^{u+h} 1+\sE \left[ \|\cX_{\cdot \wedge s}\|_\infty^q \right] \d s \leq 
          C h^{\frac{q}{2}} \left( 1 + \|\xi_{\cdot \wedge t}\|_{\bS_q}^q\right),
     \end{align*}
     where the last inequality follows from \eqref{eq:bound_X_square}.
\eproof

\subsection{Proofs in Section \ref{sec:Euler}}
\subsubsection{Proof of Lemma \ref{le prop interpol}}\label{ssec:proof_le prop interpol}   
Take $x \in \pathspace$.
For any $k = 0, 1, \cdots, N-1,$ 
since $t\rightarrow|\upl x_t |^2 = \left| x_{t_k} + \frac{t-t_k}{h} (x_{t_{k+1}} - x_{t_k})\right|^2$ is convex in $t$ on $[t_k,t_{k+1}]$,
we have
\begin{align*}
    \int_{t_k}^{t_{k+1}}| \upl x_t |^2 \d t 
        & \leq  \int_{t_k}^{t_{k+1}}
         \bigg( \frac{t_{k+1}-t}{h} |\upl x_{t_k}|^2 + \frac{t-t_k}{h} |\upl x_{t_{k+1}}|^2 \bigg) \d t  =  \frac{h}{2}  (|x_{t_k}|^2 +  |x_{t_{k+1}}|^2).
\end{align*}

By summing up the above inequality, we conclude to $(i)$.

Fix $\xi \in \bS_2(\sF)$ and $\a \in \cA$. To ease the notation, we write $\cX \coloneq \cX^{0,\xi,\a}.$
For any $t\in[0,T],$ $
    |\upl \cX_t - \cX_t| \leq \max \{ |\cX_{\n(t)} - \cX_t|, |\cX_{\n(t)+1} - \cX_t| \}.
$
Therefore, by Proposition \ref{prop:X_square_bound}, 
$
\sE\left[|\upl \cX_t - \cX_t|^2 \right] \leq C h,
$
which yields $(ii)$.
\eproof
\subsubsection{Proof of Proposition \ref{pr prop:moment_estimate_first_euler}} \label{ssec:proof_first_euler_prop}
Fix $\a \in \cA^\pi, t\in\pi, \xi \in \mathbf{S}_q(\sF).$
To ease the notation, denote $ \bar \cX = \bar \cX^{t,\xi,\a}.$

    Since the Euler update \eqref{eq de euler scheme} 
defines $\bar{\mathcal{X}}_{t_{n+1}}$ as a deterministic function of 
$\big(\bar{\mathcal{X}}_{t_n}, {\alpha}_n, \Delta W_{n+1}\big)$, 
it follows that, given the initial condition, the control sequence, 
and the Brownian path, the entire trajectory is uniquely determined. 
In particular, $\bar{\mathcal{X}}$  enjoys pathwise uniqueness, 
which in turn guarantees the flow property (cf. Proposition \ref{pr prop:moment_estimate_first_euler}(ii)).

\paragraph{Bound $\|\bar \cX_{\cdot \wedge t_n} \|_{\mathbf{S}_q}$} To ease the notation, $b_{t_n} \coloneqq b(t_n ,  \upl \bar \cX_{\cdot\wedge t_n}, \cL( \upl \bar \cX_{\cdot \wedge t_n}), \a_n),$ and $\sigma_{t_n} \coloneqq \sigma(t_n , \upl \bar  \cX_{\cdot\wedge t_n}, \cL(\upl \bar\cX_{\cdot \wedge t_n})).$
Consider the auxiliary continuous process that coincides $\{ \bar \cX_{t_n}\}_n$ on $\pi$, as defined in \eqref{eq de euler scheme cont}.

By \eqref{eq de euler scheme cont}, for any $t\leq t' \leq T$,
\begin{align*}
    \sE \left[ \sup_{0 \leq u \leq t'} \left| \bar \cX_{u } \right|^q \right] &\leq C \left(  \sE\left[ \sup_{0\leq u \leq t'} \left| \int_t^{u \vee t} b_{\n(s)} \d s\right|^q \right] + \sE\left[\sup_{u \leq t'} \left| \int_t^{u\vee t} \sigma_{\n(s)} \d W_s\right| ^q \right]  +  \sE\left[|\xi_{u \wedge t}|^q\right] \right)\\
    &\leq C   \left(  \sE\left[ \int_t^{t'}  \left| b_{\n(s)} \right|^q\d s \right] + \sE\left[ \int_t^{t'} \left\|\sigma_{\n(s)} \right\|_F ^q\d s \right]  +  \|\xi_{\cdot \wedge t}\|^q_{\mathbf{S}_q}\right) \\
    &\leq C \left( \int_t^{t'} 1+\sE\left[\left\| \bar \cX_{\cdot \wedge \n(s)} \right\|_\infty^q\right] \d s +  \| \xi_{\cdot \wedge t}\|^q_{\mathbf{S}_q}\right) 
\end{align*}
  where the second inequality follows from the Holder inequality and BDG inequality (cf. \cite[Theorem 1.111]{fabbri2017stochastic}), and the last inequality follows from Assumption \ref{ass assum:L2_lip_b_sigma} and \eqref{eq:wasserstein}.


   Thus, for $t=t_0$ and $t'=t_{n+1}$, by applying the above inequality we obtain
\(
\|\bar \cX_{\cdot \wedge t_{n+1}}\|_{\mathbf S_q}^q
\le C(1 + h \sum_{i=0}^n \|\bar \cX_{\cdot \wedge t_i}\|_{\mathbf S_q}^q
+ \|\xi_{\cdot \wedge t}\|_{\mathbf S_q}^q).
\)
We now localize with the stopping time
\(
\tau_R := \inf\bigl\{t\in[0,T]: \|\bar \cX_{\cdot\wedge t}\|_{\mathbf S_q}\ge R\bigr\}\wedge T.
\)
Applying the same estimate to the stopped process $\bar \cX_{\cdot\wedge\tau_R}$, we may use Gronwall's inequality on $[0,\tau_R]$. The resulting bound is uniform in $R$. Finally, letting $R\to\infty$ and using Fatou's lemma yields
\(
\|\bar \cX_{\cdot \wedge t_n}\|_{\mathbf S_q}^q
\le C\Bigl(1+\|\xi_{\cdot \wedge t}\|_{\mathbf S_q}^q\Bigr),
\text{ for any } n=0,1,\dots,N. 
\) 
     \eproof

\subsection{Proofs in Section \ref{sec:particle_approx}}\label{sssec:proof_of_thm_particle_approx}

\subsubsection{Proof of Lemma \ref{lemma:delta_X_B_Sigma}}
We denote $\delta X^m=\bar{\cX}^m - X^m$, for any $1\leq m\leq M$.
For each $n = 1,2,\cdots, N$, we also denote 
 $ \delta b_n^m =b\left(t_n, \bar\cX_{\cdot \wedge t_n}^{m}, \cL\left(\bar\cX_{\cdot \wedge t_n}^{m}\right), \bar\alpha^{W,m}_n\right)  -  b\left(t_n, X^m_{\cdot \wedge t_n}, {\mu}_{t_n}^M, \bar\alpha^{W,m}_n\right)$, and $
  \delta \sigma_n^m = \sigma\left(t_n, \bar\cX_{\cdot \wedge t_n}^{m}, \cL\left(\bar\cX_{\cdot \wedge t_n}^{m}\right)\right) - \sigma\left(t_n, X^m_{\cdot \wedge t_n}, {\mu}_{t_n}^M\right).$
  We also denote $  \overline{ \mathfrak{B}}_n \coloneqq \sE \left[ b\left( t_n, \bar\cX^m_{\cdot \wedge t_n},\cL (\bar\cX^m_{\cdot \wedge t_n}), \bar\alpha^{W,m}_n\right)-b\left(t_n, \bar\cX^m_{\cdot \wedge t_n}, \bar\mu_{t_n}^M, \bar\alpha^{W,m}_n\right) \right]$, and $\overline{\mathfrak{S}}_n$ similarly.

Then we have $\delta X_{t_{n+1}}^m=\delta X_{t_n}^m + \delta b_n^m h+\delta \sigma_n^m \Delta W_{n+1}^m,$ and
\begin{equation}\label{eq:delta_X}
\begin{aligned}
     \mathbb{E}[|\delta X_{t_{n+1}}^m|^2]
&= \sE[ |\delta X_{t_n}^m|^2 ]
+ 2 h \sE \left[ \left(\delta X_{t_n}^m \right)^\top \delta b_n^m\right]
+ h^2 \sE \left[  |\delta b_n^m|^2 \right] + \sE [|\delta \sigma_n ^m \Delta W_{n+1}|^2] \\
     &\leq (1+ h)\mathbb{E}\left[\left|\delta X_{t_n}^m\right|^2\right]+  ( h+h ^2) \mathbb{E}\left[\left| 
     \delta b_n^m\right|^2 \right] +  {h}\sE \left[ \left\|\delta \sigma_n^m \right\|_F^2 \right]. 
\end{aligned}
\end{equation}

Now, note that
$$
\begin{aligned}
&\sE\left[ \left|\delta b_n^m\right| ^2\right]
 = \sE \left[ \left|b\left(t_n, \bar\cX_{\cdot \wedge t_n}^{m}, \cL\left(\bar\cX_{\cdot \wedge t_n}^{m}\right), \bar\alpha^{W,m}_n\right)  -  b\left(t_n, X^m_{\cdot \wedge t_n}, {\mu}_{t_n}^M, \bar\alpha^{W,m}_n\right)\right| ^2 \right]\\
& \leq 2 \left(  \overline{\mathfrak{B}}_n + \sE \left[ \left|b\left(t_n, \bar\cX^m_{\cdot \wedge t_n}, \bar\mu_{t_n}^M, \bar\alpha^{W,m}_n\right) -b\left(t_n, X^m_{\cdot \wedge t_n}, \mu_{t_n}^M, \bar\alpha^{W,m}_n\right)\right|^2 \right] \right).
\end{aligned}
$$
By Assumption \ref{ass assum:L2_lip_b_sigma}, we know that the second term in the above inequality is upper bounded by
$$
\begin{aligned}
   C \sE \left[ \frac{1}{N}\sum_{i=0}^n \left|\delta X_{t_i}^m\right| ^2 +\frac{N-n}{N} |\delta X_{t_n}^m|^2 + \cW_2 \left(\bar\mu_{t_n}^M, \mu_{t_n}^M\right)^2\right].
\end{aligned}
$$
Moreover,
$$
\begin{aligned}
      \cW_2 \left(\bar\mu_{t_n}^M, \mu_{t_n}^M\right)^2& \leq \sE\left[ 
  \frac{1}{M} \sum_{k=1}^M \|\bar\cX^k_{\cdot \wedge t_n} -  X^k_{\cdot \wedge t_n} \|_{L^2(t_n)}^2\right] \\
  &\leq   \frac{1}{ M}  \sum_{k=1}^M \left(\frac{1}{N} \sum_{i=0}^n\sE \left[ |\delta X_{t_i}^k |^2 \right] + \frac{N-n}{N}\sE  \left[ |\delta X_{t_n}^k |^2 \right]\right)
\end{aligned}
$$
 We then conclude that
\begin{equation}\label{eq:delta_b}
\begin{aligned}
     \mathbb{E}\left[\left|\delta b_n^m\right|^2\right]
    &\leq C \left(  \overline{\mathfrak{B}}_{n} + 
  \frac{1}{N} \sum_{i=0}^n\sE \left[ |\delta X_{t_i}^k |^2 \right] + \frac{N-n}{N}\sE  \left[ |\delta X_{t_n}^k |^2  \right] \right) ,
\end{aligned}
\end{equation}
by noting that $\{\delta X^k\}_{k=1}^M$ are i.i.d.
Similar arguments lead to
\begin{equation}\label{eq:delta_sigma}
     \mathbb{E}\left[\left\|\delta \sigma_n^m\right\|_F^2\right] \leq
     C \left( \overline{\mathfrak{S}}_{n} +
  \frac{1}{N} \sum_{i=0}^n\sE \left[ |\delta X_{t_i}^k |^2 \right] + \frac{N-n}{N}\sE  \left[ |\delta X_{t_n}^k |^2 \right] 
  \right).
\end{equation}
Combining \eqref{eq:delta_b}, \eqref{eq:delta_sigma} with \eqref{eq:delta_X}, one obtains
\begin{equation*}
    \begin{aligned}
         \mathbb{E}\left[\left|\delta X_{t_{n+1}}^m\right|^2\right] &\leq 
         (1+ C h)\mathbb{E}\left[\left|\delta X_{t_n}^m\right|^2\right] 
    +  Ch  \left( \overline{\mathfrak{B}}_n + \overline{\mathfrak{S}}_n
   +   \frac{1}{N} \sum_{i=0}^n\sE \left[ \left| \delta X_{t_i}^m \right|^2 \right] 
   \right).
\end{aligned}
\end{equation*}
Taking maximum over the both sides, and denoting $y_n = \max_{i\leq n} \sE[|\delta X_{t_{i}}^m|^2]$,
we then have
$
y_{n+1} \leq y_n + Ch \left(\max_{i\leq n} (\overline{\mathfrak{B}}_i + \overline{\mathfrak{S}}_i  )+ y_n\right).
$
By Gronwall's inequality, 
$y_n \leq C\max_{i\leq n} (\overline{\mathfrak{B}}_i + \overline{\mathfrak{S}}_i  ). $ The proof is finished by noting Assumption \ref{ass assum:L2_lip_b_sigma}. 
\eproof

\subsubsection{Proof of Theorem \ref{thm:particle_approx}}
 Fix a control $\psi^W \in \cM^W$ and let $(\bar\alpha^{W,m}_n)_{n \geq 0}$ be given by \eqref{eq:bar_alpha_W_particle}. To ease the notation, we write $\bar X^{\xi^m,W}$ as $\bar{X}^m$, and $X^{\xi^m,W}$ as $X^m.$

Note that for any $1 \leq m\leq M$,
\begin{equation}\label{eq:g_minus_bar_g}
        \begin{aligned}
        & \sE \left[ \left| \Terminal\left(X^{m},  \mu_{t_N}^M\right)  - \Terminal\left(\bar\cX^{m}, \mathcal{L} \left(\bar\cX^{m}\right)\right)\right| \right]\\
        &\leq \sE \left[ \left| \Terminal( X^m, \cL(\bar\cX^m)) - \Terminal(
        \bar\cX^m, \cL(\bar\cX^m))\right|\right] 
        + \sE \left[ \left| \Terminal(  X^m, \cL( \bar\cX^m)) - \Terminal(
         X^m,   \mu_{t_N}^M)\right|\right],
        \end{aligned}
        \end{equation}
        where the first term can be bounded by Lemma \ref{lemma:delta_X_B_Sigma}, 
        Proposition \ref{pr prop:moment_estimate_first_euler} and \eqref{eq:moment_bound_particle_approx},
        \begin{equation}\label{eq:g_minus_bar_g1}
            \begin{aligned}
            &  \sE \!\left[ \left| \Terminal( X^m, \cL(\bar\cX^m)) - \Terminal(
        \bar\cX^m, \cL(\bar\cX^m))\right|\right]\\
        &\leq C \Big(1+ \sE\left[\|X^m\|_\infty^2\right]^{\frac{1}{2}}
          + \sE\left[\|\bar\cX^m \|_\infty^2 \right]^{\frac{1}{2}} \Big)
          \sE \!\left[ \left\| X^m - \bar\cX^m\right\|_{2} ^2\right]^{\frac{1}{2}} \\
        &\leq C \max_{1\leq n \leq N}
          \sE\!\left[ \cWtwo\left(\cL (\bar\cX^m_{\cdot \wedge t_n}), \bar\mu_{t_n}^M\right)^2 \right]^{\frac{1}{2}}
          \leq C M^{- \frac{\gamma(q,\tilde{d})}{2}},
            \end{aligned}
        \end{equation}
      where the last inequality follows from Lemma \ref{lemma:wasserstein_bound} with $\tilde {d} \coloneqq (N+1) d$.
To bound the second term in \eqref{eq:g_minus_bar_g},
by Assumption \ref{assum:terminal_cost_f_g_Phi}, 
        \begin{equation*}
      \begin{aligned}
   &\sE \left[ \left| \Terminal( X^m, \cL( \bar\cX^m)) - \Terminal(
         X^m,   \mu_{t_N}^M)\right|\right]\\   
    &\leq C 
        \sE \left[ \cW_{2, L^2}\left(\cL({\bar\cX^m}), \mu_{t_N}^M\right)\cdot \left(1+ \|X^m \|_{\infty} + \cWinf(\cL(X^m), \delta_0 ) + \cWinf(\mu^M_{t_N}, \delta_0)\right)\right]\\
        &\leq  C
\left( \sE \left[\cW_{2, L^2}\left(\cL({\bar\cX^m}),  \bar\mu_{t_N}^M\right) ^2 \right]^{\frac{1}{2}} + \sE \left[\cW_{2, L^2}\left(\bar\mu_{t_N}^M, \mu_{t_N}^M\right) ^2 \right]^{\frac{1}{2}}\right) \\
&\qquad
        \cdot\sE\Big[ 1+ \|X^m \|_{\infty}^2 + \cWinf(\cL(X^m), \delta_0 )^2 + \cWinf(\mu^M_{t_N}, \delta_0)^2 \Big]^\frac{1}{2} \\
& \overset{(a)}{\leq} C \Big( \sE \left[\cWtwo\left(\cL({\bar\cX^m}),  \bar\mu_{t_N}^M\right) ^2 \right]^{\frac{1}{2}} + \sE \Big[ \frac{1}{M} \sum_{k=1}^M \left\|X^k - \bar\cX^k \right\|^2_2 \Big]^{\frac{1}{2}}\Big)  \\
& \overset{(b)}{\leq} C 
\sE \!\left[\cWtwo\left(\cL({\bar\cX^m}),  \bar\mu_{t_N}^M\right) ^2 \right]^{\frac{1}{2}}
+ C\max_{n \in [N]} \sE\!\left[ \cWtwo\left(\cL (\bar\cX^m_{\cdot \wedge t_n}), \bar\mu_{t_n}^M\right)^2 \right]^{\frac{1}{2}}
\overset{(c)}{\leq} C M^{- \frac{\gamma(q,\tilde{d})}{2}}. 
    \end{aligned}
    \end{equation*}
where 
(a) follows from \eqref{eq:moment_bound_particle_approx}, 
(b) follows from Lemma \ref{lemma:delta_X_B_Sigma}, and (c) follows from Lemma \ref{lemma:wasserstein_bound}.

  Similar bound also holds for $ \overline{\mathfrak{S}}_n.$
       By plugging in the above inequality and \eqref{eq:g_minus_bar_g1} to \eqref{eq:g_minus_bar_g}, we have
    $
    \sE \left[ \left| \Terminal\left(X^{m},  \mu_{t_N}^M\right)  - \Terminal\left(\bar\cX^{m}, \mathcal{L} \left(\bar\cX^{m}\right)\right)\right| \right] \leq CM^{- \frac{\gamma(q,\tilde{d})}{2}}.
    $
Note that similar estimate also holds for the running cost term $f$. Combining together, we have $|\bar J(0, \mu, \fb) - J^M (0,\mu,\fb) | \leq C M^{-\frac{\gamma(q,\tilde d)}{2}} $.

    Finally, fix $\vep > 0.$ Denote $\psi^W_1$ and $\psi^W_2$ to be the $\vep$-optimal solution to $\bar J$ and $J^M$ respectively. Then
  $$
   V^M_0(\mu) \leq J^M(0,\mu,\psi^W_1) \leq  \bar J(0,\mu,\psi^{W_1}) + C M^{-\gamma(q,\tilde{d})/2}\leq {\widehat{V}(0,\mu)} +  CM^{-\gamma(q,\tilde{d})/2}+ \vep.
  $$
  Similarly, 
  $$
    \widehat{V}(0,\mu) \leq\bar J(0,\mu,\psi^{W_2})\leq  J^M(0,\mu,\psi^W_2)+ C M^{- \gamma(q,\tilde{d})/2} \leq V^M_0(\mu)  + C  M^{-\gamma(q,\tilde{d})/2} +\vep.
  $$
  The proof is finished by combing the above two inequalities and taking $\vep \downarrow 0$.  
\eproof

\bibliographystyle{alpha}
\bibliography{reference}

\appendix

\section{LQ example with analytical solution}\label{appendix:LQ}
Let $Z_t=(X_t,Y_t)^\top\in\mathbb R^{2d}$ with $Y_t =\int_0^t e^{-\lambda(t-s)} X_s \, \d s,$ and   
\(
\d X_t=(aX_t+b\alpha_t+c\,\mathbb E[X_t])\d t+\sigma\d W_t.\)
Hence,
\[
dZ_t=(AZ_t+B\alpha_t+C\,\mathbb E[Z_t])\,dt+\Sigma\,dW_t, \quad Z_0 = (X_0, 0)^\top,
\]
with
\(
A=\begin{bmatrix}aI_d&0\\ I_d&-\lambda I_d\end{bmatrix},
B=\begin{bmatrix}bI_d\\0\end{bmatrix},
C=\begin{bmatrix}cI_d&0\\0&0\end{bmatrix}, \text{ and }
\Sigma=\begin{bmatrix}\sigma\\0\end{bmatrix}.
\)
Denote $\bar \mu = \int z \mu(\d z)$ and 
consider the quadratic costs
\(f(t,z,\mu,\alpha)=\tfrac12\!\left(z^\top Q_t z+\bar\mu^\top P_t\bar\mu+\alpha^\top R_t\alpha\right)+v_t^\top z, \text{ and }
g(z,\mu)=\tfrac12\!\left(z^\top \bar Q z+\bar\mu^\top \bar P\bar\mu\right)+\bar v^\top z,
\)
where $Q_t,P_t,\bar Q,\bar P,R_t$ are symmetric and $R_t\succ0$. 
\begin{theorem}
    Define $w: [0,T] \times \cP(\sR^{2d})\ra \sR$ as
  $  w(t,\mu)=\tfrac12\!\int_{\sR^{2d}} z^\top M_t z\,\mu(\d z)+\tfrac12\bar\mu^\top N_t\bar\mu+\rho_t^\top\bar\mu+h_t.$
If $(M_t,N_t,\rho_t,h_t)$ solve the Riccati system
\begin{equation}\label{eq:ODE_lq}
    \begin{cases}
0= \dot M + A^\top M_t+M_tA+Q_t-M_tBR_t^{-1}B^\top M_t,\\
0=\dot N_t + (A+C)^\top N_t+N_t(A+C)+C^\top M_t+M_tC+P_t\\
\qquad\quad -N_tBR_t^{-1}B^\top N_t-M_tBR_t^{-1}B^\top N_t-N_tBR_t^{-1}B^\top M_t,\\
0=\dot\rho_t + (A+C)^\top\rho_t+v_t-(M_t+N_t)BR_t^{-1}B^\top\rho_t,\\
0=\dot h_t + \tfrac12\mathrm{Tr}(\Sigma\Sigma^\top M_t)-\tfrac12\rho_t^\top BR_t^{-1}B^\top\rho_t,
\end{cases}
\end{equation}
with terminal data $M_T=\bar Q$, $N_T=\bar P$, $\rho_T=\bar v$, $h_T=0$, then $w$ solves the HJB equation and equals the value function:
\(
V(t,\mu)=w(t,\mu).
\)
The optimal feedback control is
\begin{equation}\label{eq:alpha_LQ}
    \alpha^*(t,z,\mu)=-R_t^{-1}B^\top\big(M_tz+N_t\bar\mu+\rho_t\big).
\end{equation}
\end{theorem}

\proof
The Lions derivatives of $w$ are given by
\[
\partial_\mu w(t,\mu)(z)
=
M_tz+N_t\bar\mu+\rho_t,
\qquad
\partial_z\partial_\mu w(t,\mu)
=
M_t.
\]
We define the Hamiltonian $H$ by
\[
\begin{aligned}
H(t,z, \mu,\a)
&=
-\partial_t w(t,\mu)
-
\mu\Big(
(Az+C\bar\mu)^\top
(M_tz+N_t\bar\mu+\rho_t)
\Big)
-
\frac12
\operatorname{Tr}(\Sigma\Sigma^\top M_t)
\\
&\quad
-
\frac12
\int_{\sR^{2d}} z^\top Q_tz\,\mu(\d z)
-
\frac12
\bar\mu^\top P_t\bar\mu
-
v_t^\top\bar\mu
\\
&\quad
-
\mu\Big(
\alpha^\top B^\top
(M_t z+N_t\bar\mu+\rho_t)
+
\frac12\alpha^\top R_t\alpha
\Big) .
\end{aligned}
\]
The supremum of $H$ over $A$ is attained at $\a^*$ in \eqref{eq:alpha_LQ}. With the coefficient of $w$ satisfying \eqref{eq:ODE_lq}, it can be verified that $H(t,z,\mu,
\a^*(t,z,\mu)) = 0.$

Fix $\a \in \cA,$ an initial time $t$ and an initial distribution $\mu.$ 
By Ito's formula on flows of measure \cite{Guo2023ito}., 
we have
\begin{align*}
    w(T, \cL(Z_T^\a)) -  w(t, \mu) =&\sE \Big[\int_{t}^T \partial_t w(s,\cL(Z_s^\a)) + (M_s Z_s^\a + N_s \sE[Z_s^\a] + \rho_s)^\top (A Z_s^\a + B \a + C \sE[Z_s^\a])  \\
    & \quad + \frac{1}{2} \Tr(\Sigma\Sigma^\top M_s)\Big] \geq \sE [\int_t^T - f(s, Z_s^\a, \cL(Z_s^\a), \a_s)]
\end{align*}
Therefore,\[
w(t,\mu)
\leq
\mathbb E
\left[
\int_t^T
f(s,Z_s^\alpha,\cL(Z_s^\alpha),\alpha_s)\,\d s
+
g(Z_T^\alpha,\cL(Z_T^\alpha))
\right]
=
J(t,\mu,\alpha).
\]
Since \(\alpha\in\mathcal A\) was arbitrary, we get
\(
w(t,\mu)\leq V(t,\mu).
\)
For \(\alpha=\alpha^*\), the equality is attained for the above argument. Hence
\(
w(t,\mu)=V(t,\mu).
\)

\end{document}